\documentclass[a4paper, svgnames, 10pt]{amsart}

\usepackage[a-1b]{pdfx}

\usepackage{
    amssymb,
    mathtools,
    xcolor,
    booktabs,
    multicol,
    bm,
    customdice,
}

\usepackage[english]{babel}

\usepackage{enumitem}
\setlist[enumerate]{label=\textup{(\arabic*)}}

\usepackage{hyperref}
\hypersetup{colorlinks, allcolors=MediumBlue}
\newcommand{\msc}[1]{\href{https://zbmath.org/classification/?q=#1}{#1}}
\newcommand{\doi}[1]{doi:\href{https://doi.org/#1}{#1}}

\usepackage{orcidlink}

\usepackage[noabbrev, capitalise]{cleveref}
\crefdefaultlabelformat{#2\textup{#1}#3}

\crefname{main} {Theorem}       {Theorems}
\crefname{thm}  {Theorem}       {Theorems}
\crefname{lem}  {Lemma}         {Lemmas}
\crefname{prop} {Proposition}   {Propositions}
\crefname{dfn}  {Definition}    {Definitions}
\crefname{fig}  {Figure}        {Figures}
\crefname{tbl}  {Table}         {Tables}
\crefname{rmk}  {Remark}        {Remarks}
\crefname{exm}  {Example}       {Examples}

\theoremstyle{plain}
\newtheorem{main} {Theorem}

\newtheorem{thm} {Theorem} [section]
\newtheorem{prop}   [thm] {Proposition}
\newtheorem{cor}    [thm] {Corollary}
\newtheorem{lem}    [thm] {Lemma}
\theoremstyle{definition}

\newtheorem{rmk}    [thm] {Remark}

\numberwithin{equation}{section}

\newcommand{\F}{\mathbb{F}}
\newcommand{\FD}{\mathfrak{D}}
\newcommand{\FC}{\mathfrak{C}}

\DeclarePairedDelimiterX\set[1]\lbrace\rbrace{\,\def\given{\mid}#1\,}
\DeclarePairedDelimiterX\gen[1]\langle\rangle{\,\def\given{\mid}#1\,}
\newcommand{\Agemo}{\mho}                    
\DeclareMathOperator{\Frat}{\Phi}            
\newcommand{\Cyc}[1]{C_{#1}}                 
\newcommand{\D}[1]{#1'}                      
\newcommand{\Z}{\bm{\mathrm{Z}}}             
\newcommand{\CC}{\bm{\mathrm{{C}}}}          
\newcommand{\Soc}[1]{\Omega(\Z(#1))}         
\newcommand{\DG}[1]{G_{\scriptsizedice{#1}}} 

\makeatletter
\@namedef{subjclassname@2020}{%
    \textup{2020} Mathematics Subject Classification}
\makeatother

\begin{document}

\title[Identification of non-isomorphic $2$-groups]
    {Identification of non-isomorphic $2$-groups with\\ dihedral central quotient and\\ isomorphic modular group algebras}

\author[L.~Margolis]{Leo Margolis\,\orcidlink{0000-0002-9413-0775}}
\address[Leo Margolis]
    {Universidad Aut\'onoma de Madrid, Departamento de Matem\'aticas,  C/ Francisco Tom\'as y Valiente, 7 Facultad de Ciencias, m\'odulo 17, 28049 Madrid, Spain.}
\email{leo.margolis@icmat.es}

\author[T.~Sakurai]{Taro Sakurai\,\orcidlink{0000-0003-0608-1852}}
\address[Taro Sakurai]
    {Department of Mathematics and Informatics, Graduate School of Science, Chiba University, 1-33, Yayoi-cho, Inage-ku, Chiba-shi, Chiba, 263-8522, Japan.}
\email{tsakurai@math.s.chiba-u.ac.jp}




\date{\today}

\begin{abstract}
    The question whether non-isomorphic finite $p$-groups can have isomorphic modular group algebras was recently answered in the negative by
    Garc\'{\i}a-Lucas, Margolis and del R\'{\i}o [J.\ Reine Angew.\ Math.\ 783 (2022), pp.~269--274].
    We embed these negative solutions in the class of two-generated finite $2$-groups with dihedral central quotient,
    and solve the original question for all groups within this class.
    As a result, we discover new negative solutions and simple algebra isomorphisms.
    At the same time, the positive solutions for most of the groups in this class give some insights what makes the negative solutions special.
\end{abstract}

\maketitle


\section*{Introduction}
The \emph{modular isomorphism problem} asks whether non-isomorphic finite $p$-groups can have isomorphic group algebras over a field of positive characteristic~$p$.
It seems to have first been raised in the 1950s and appears in an influential survey by Brauer~\cite[Problem 2 and Supplement \S 9]{Bra63}.
Over the following decades, positive results for many classes of finite $p$-groups were published, cf.~\cite{Mar22} for an overview.
However, a few years ago, Garc\'{\i}a-Lucas, Margolis and del R\'{\i}o~\cite{GLMdR22} discovered
the first non-isomorphic finite $2$-groups that have isomorphic group algebras over an arbitrary field of characteristic~$2$.

Nevertheless, the problem remains poorly understood in general and open for odd primes, for instance.
Our goal in this article is to understand better how the negative solutions presented in~\cite{GLMdR22} come to be.
These groups were discovered during an effort to solve the modular isomorphism problem for two-generated finite $p$-groups with cyclic derived subgroup,
after these had been classified by Broche, Garc\'{\i}a-Lucas and del R\'{\i}o~\cite{BGLdR23}.
Although considerable effort has been devoted to solving the problem in this wide class,
it has not been achieved completely so far~\cite{GLdRS23, GLdR23}.
We consider the negative solutions from another viewpoint
and embed them in a natural class of groups for which we can completely solve the modular isomorphism problem.

We call a non-abelian group that is generated by two elements of order~$2$ \emph{dihedral}.
For a group $G$, we call the quotient by its center the \emph{central quotient} of $G$.
In this article, we study the modular isomorphism problem for two-generated finite $2$-groups with dihedral central quotient.
This class of groups includes the negative solutions presented in~\cite{GLMdR22},
and we are able to identify all negative solutions to the problem---including new counterexamples---within this class of groups.

\newpage 
\begin{main}\label{main:A}
    Let $G$ and $H$ be finite $2$-groups and $\F$ a field of characteristic~$2$.
    Suppose that $G$ is two-generated and the central quotient of $G$ is dihedral.
    Then $\F G \cong \F H$ but $G \not\cong H$ if and only if $G$ and $H$ are isomorphic to the groups
    \begin{align*}
        \langle\, x, y, z \mid x^{2^n} &= 1, & y^{2^m} &= 1,       & z^{2^\ell} &= 1, & [y, x] &= z, & [z, x] &= z^{-2}, & [z, y] &= z^{-2} \,\rangle, \\
        \langle\, a, b, c \mid a^{2^n} &= 1, & b^{2^m} &= a^{2^m}, & c^{2^\ell} &= 1, & [b, a] &= c, & [c, a] &= c^{-2}, & [c, b] &= c^{-2} \,\rangle
    \end{align*}
    for some integers $n > m > \ell \ge 2$.
\end{main}

The known counterexamples to the modular isomorphism problem presented in~\cite{GLMdR22} correspond to the case $\ell = 2$.
It should be emphasized, however, that our algebra isomorphisms are simpler than the original ones and it makes the proof cleaner.
This is largely due to the change of presentations of the groups, which is natural in our class.

For the proof of \cref{main:A}, we first give presentations of the groups in our target class.
We then establish an isomorphism between the algebras of the above groups in \cref{thm:IsomorphicAlgebras}.
The proof of this isomorphism can be read independently of the rest of this article,
as it only needs the presentations of the groups.
We remark that the proof can not be carried over to odd primes,
or at least not in the most naive way (cf. \cref{rmk:OddPrimes}).
Finally, we prove that no more negative solutions exist within the class of two-generated finite $2$-groups with dihedral central quotient.
This is the most challenging part of the article, and what can be considered standard methods to attack the problem turn out to be insufficient to prove this.
We hence introduce a procedure that could be called a ``group base approximation''.
It allows us, after considerable effort, to solve the modular isomorphism problem completely within our class.
This procedure could also be used for other classes of groups with some adjustments.

As a byproduct of identifying negative solutions, we obtain the classification of two-generated finite $2$-groups with dihedral central quotient.
Since such a classification of groups might be also of independent interest, we provide a direct proof in~\cref{sec:Appendix}.

\section{Preliminaries}
We use standard group-theoretical notation.
We always write $G$ and $H$ for finite $p$-groups for some prime~$p$.
In most cases, $p = 2$ holds.
The center of $G$ is denoted by $\Z(G)$, the derived subgroup by $\D{G}$ and the Frattini subgroup by $\Frat(G)$.
By $\Omega(G)$ we mean the subgroup of $G$ generated by the elements of order~$p$.
For a non-negative integer $r$ we write $\Agemo_r(G)$ for the subgroup of $G$ generated by $g^{p^r}$, $g \in G$.
For $g, h \in G$ we let $[g, h] = g^{-1}h^{-1}gh$ denote the commutator.
Accordingly we conjugate as $g^h = h^{-1}gh$.
We write a cyclic group of order $n$ as $\Cyc{n}$.

Let $\F$ denote a field of characteristic~$p$ and $\F G$ the group algebra of $G$ over $\F$.
We write $\Delta(G)$ for the augmentation ideal in the group algebra $\F G$.
Note that $\Delta(G)$ is nilpotent and equal to the Jacobson radical of $\F G$.
In particular, the complement of $\Delta(G)$ equals the set of units in $\F G$.
The center of an algebra $\Lambda$ over $\F$ is denoted by $\Z(\Lambda)$, and
the linear subspace spanned by the Lie commutators is denoted by $[\Lambda, \Lambda]$.
We note that $\Delta(\D{G})\F G$ is the relative augmentation ideal of the derived subgroup
and equal to the smallest ideal $[\F G, \F G]\F G$ of $\F G$ with commutative quotient.
Moreover, we let $\Agemo_r(\Lambda)$ be the subalgebra of $\Lambda$ generated by $T^{p^r}$, $T \in \Lambda$.

The center of $\F G$ has a special decomposition which we will use frequently.
Namely $\Z(\F G) = \F\Z(G) \oplus (\Z(\F G) \cap [\F G, \F G])$.
Moreover, $\Z(\F G) \cap [\F G, \F G]$ coincides with the linear subspace of $\Z(\F G)$ spanned by the class sums of non-central elements of $G$.
We refer to~\cite[Section III.6]{Seh78} for more details and proofs.
We will allude to these facts in \cref{sec:PositiveSolutions} without further mention.

A property of $G$ is an \emph{invariant} if $\F G \cong \F H$ implies the same property for $H$.
A subset of $\F G$ is called \emph{canonical} if it is stable under all algebra automorphisms.

\section{Presentations of groups}

In this section, we give presentations of two-generated finite $2$-groups with dihedral central quotient.
Compared to many other classes of finite $p$-groups, this class turns out to be rather accessible.

\begin{thm}\label{thm:Presentations}
    Let $G$ be a two-generated finite $2$-group with dihedral central quotient.
    Then $G/\D{G} \cong \Cyc{2^n} \times \Cyc{2^m}$ and $\D{G} \cong \Cyc{2^\ell}$ for some positive integers $n$, $m$ and $\ell$ with $n \ge m$ and $\ell \ge 2$,
    and $G$ is isomorphic to a group generated by $x$, $y$ and $z$ with one of the following defining relations.
    \begin{align*}
        \dice{1}: & & x^{2^n} &= 1,                \quad y^{2^m} = 1,                        & z^{2^\ell} &= 1, & [y, x] &= z, & [z, x] &= z^{-2}, & [z, y] &= z^{-2}; \\
        \dice{2}: & & x^{2^n} &= 1,                \quad y^{2^m} = x^{2^m},                  & z^{2^\ell} &= 1, & [y, x] &= z, & [z, x] &= z^{-2}, & [z, y] &= z^{-2}; \\
        \dice{3}: & & x^{2^n} &= 1,                \quad y^{2^m} = z^{2^{\ell - 1}},         & z^{2^\ell} &= 1, & [y, x] &= z, & [z, x] &= z^{-2}, & [z, y] &= z^{-2}; \\
        \dice{4}: & & x^{2^n} &= 1,                \quad y^{2^m} = x^{2^m} z^{2^{\ell - 1}}, & z^{2^\ell} &= 1, & [y, x] &= z, & [z, x] &= z^{-2}, & [z, y] &= z^{-2}; \\
        \dice{5}: & & x^{2^n} &= z^{2^{\ell - 1}}, \quad y^{2^m} = 1,                        & z^{2^\ell} &= 1, & [y, x] &= z, & [z, x] &= z^{-2}, & [z, y] &= z^{-2}; \\
        \dice{6}: & & x^{2^n} &= z^{2^{\ell - 1}}, \quad y^{2^m} = x^{2^m},                  & z^{2^\ell} &= 1, & [y, x] &= z, & [z, x] &= z^{-2}, & [z, y] &= z^{-2}.
    \end{align*}
\end{thm}

We write
\begin{equation}\label{eq:DiceGroups}
    \DG{1}, \DG{2}, \dotsc, \DG{6}
\end{equation}
for the groups of order $2^{n + m + \ell}$ in \cref{thm:Presentations}
and always use the generators satisfying the above relations, which are written $x, y, z$ or $a, b, c$.

Evidently, $\DG{1} \cong \DG{2}$ and $\DG{3} \cong \DG{4} \cong \DG{5}$ if $n = m$.
We will see in the end that all other pairs of groups are non-isomorphic through investigation of their modular group algebras.
Even though such indirect arguments suffice to prove \cref{main:A}, a more direct proof is also of interest.
Such a proof is presented in \cref{sec:Appendix}.

We also note that the groups \eqref{eq:DiceGroups} include finite $2$-groups of maximal class:
the dihedral, semidihedral and generalized quaternion groups (\cref{prop:MaximalClass}).

\begin{lem}\label{lem:Cyclic}
    Let $G$ be a group with dihedral central quotient of order $2^{\ell + 1}$.
    Then the derived subgroup $\D{G}$ is cyclic of order $2^\ell$.
\end{lem}

\begin{proof}
    First note that $\ell \ge 2$ as $G/\Z(G)$ is non-abelian.
    Since the central quotient $G/\Z(G)$ is isomorphic to the dihedral group of order $2^{\ell + 1}$,
    the quotient $\D{G}\Z(G)/\Z(G)$, which is the derived subgroup of $G/\Z(G)$, has order $2^{\ell - 1}$.
    There are two elements $x$ and $y$ of $G$ such that $x^2$, $y^2$ and $(yx)^{2^\ell}$ are trivial modulo the center $\Z(G)$.
    Moreover, $x$ and $y$ together with the center of $G$ generate the whole group $G$.
    We write $z$ for the commutator $[y, x]$.
    Since $y^2$ is central, we have
    \[
        1 = [y^2, x] = [y, x]^y [y, x] = z^y z
    \]
    and the action of $y$ on $z$ by conjugation equals the inversion.
    The same is true for $x$.
    Then a direct calculation of commutators shows that the derived subgroup $\D{G}$ is a cyclic group generated by $z$.
    Since $z$ is congruent to $(yx)^2$ modulo the center $\Z(G)$, the power $z^{2^{\ell - 1}}$ is central.
    Then $z^{2^{\ell - 1}}$ is equal to its inverse by conjugation.
    This shows that $\D{G} \cap \Z(G)$ has order $2$ and hence $\D{G}$ has order $2^\ell$.
\end{proof}

\begin{lem}\label{lem:Parametrization}
    Let $G$ be a two-generated finite $2$-group with dihedral central quotient.
    Then $G/\D{G} \cong \Cyc{2^n} \times \Cyc{2^m}$ and $\D{G} \cong \Cyc{2^\ell}$
    for some positive integers $n$, $m$ and $\ell$ with $n \ge m$ and $\ell \ge 2$,
    and $G$ is isomorphic to the group
    \[
        G_\theta
        = \Biggl\langle\, x, y, z \Biggm|
            \begin{gathered}
                x^{2^n} = z^{r2^{\ell - 1}},\ y^{2^m} = x^{s2^m} z^{t2^{\ell - 1}},\ z^{2^\ell} = 1, \\
                [y, x] = z,\ [z, x] = z^{-2},\ [z, y] = z^{-2}
            \end{gathered}
        \,\Biggr\rangle
    \]
    for some $\theta = (r, s, t)$ belonging to
    \[
        \Theta = \set{ (r, s, t) \given 0 \le r \le 1,\ \  0 \le s \le 2^{n - m} - 1,\  \ 0 \le t \le 1 }.
    \]
\end{lem}

\begin{proof}
    By \cref{lem:Cyclic}, the derived subgroup $\D{G}$ must be cyclic and the central quotient $G/\Z(G)$ must have order $2^{\ell + 1}$.
    From the proof of \cref{lem:Cyclic}, there are three elements $x$, $y$ and $z$ of $G$ such that $x$ and $y$ generate $G$ and
    \[
        [y, x] = z, \ z^{2^\ell} = 1,\ [z, x] = z^{-2},\ [z, y] = z^{-2}.
    \]
    We may assume that $x^{2^n}$ belongs to $\D{G}$, while $x^{2^{n - 1}}$ does not.
    (Otherwise replace $x$, $y$ and $z$ by $y$, $x$ and $z^{-1}$).
    Since $x^2$ is central, we see that $x^{2^n}$ belongs to $\D{G} \cap \Z(G)$ which is a cyclic group of order $2$ generated by $z^{2^{\ell - 1}}$.
    Hence
    \[
        x^{2^n} = z^{r 2^{\ell - 1}}
    \]
    for some $0 \le r \le 1$.

    Let $\bar{G}$ denote the abelianization $G/\D{G}$.
    From~\cite[Theorem 7.12]{Isa09} there is an element $w$ of $G$ with $\bar{G} = \langle \bar{x} \rangle \times \langle \bar{w} \rangle$,
    in particular $\bar{w}$ has order $2^m$.
    Then $\bar{y}^{2^m} = \bar{x}^{s2^m}$ for some $0 \le s \le 2^{n - m} - 1$.
    Since $y^{-2^m} x^{s 2^m }$ belongs to $\D{G} \cap \Z(G)$, we obtain
    \[
        y^{2^m} = x^{s 2^m} z^{t 2^{\ell - 1}}
    \]
    for some $0 \le t \le 1$.

    The six relations above define a finite group of order $2^{n + m + \ell}$ and yield a presentation of $G$.
\end{proof}

We consider a partition of the parameter space introduced in \cref{lem:Parametrization} as $\Theta = \Theta_{\scriptsizedice{1}} \cup \dotsb \cup \Theta_{\scriptsizedice{6}}$ defined by the following.

\begin{align*}
    \Theta_{\scriptsizedice{1}} &= \set{ (r, s, t) \in \Theta \given r = 0,\ t = 0,\ s \equiv 0 \mod 2}; \\
    \Theta_{\scriptsizedice{2}} &= \set{ (r, s, t) \in \Theta \given r = 0,\ t = 0, \ s \equiv 1 \mod 2}; \\
    \Theta_{\scriptsizedice{3}} &= \set{ (r, s, t) \in \Theta \given r = 0,\ t = 1,\ s \equiv 0 \mod 2}; \\
    \Theta_{\scriptsizedice{4}} &= \set{ (r, s, t) \in \Theta \given r = 0,\ t = 1,\ s \equiv 1 \mod 2}; \\
    \Theta_{\scriptsizedice{5}} &= \begin{cases}
        \set{ (r, s, t) \in \Theta \given r = 1,\ s \equiv 0 \mod 2}                                & (n > m) \\
        \set{ (r, s, t) \in \Theta \given r = 1,\ s \equiv 0 \mod 2} \smallsetminus \{ (1, 0, 1) \} & (n = m);
    \end{cases} \\
    \Theta_{\scriptsizedice{6}} &= \begin{cases}
        \set{ (r, s, t) \in \Theta \given r = 1,\ s \equiv 1 \mod 2}                      & (n > m) \\
        \set{ (r, s, t) \in \Theta \given r = 1,\ s \equiv 1 \mod 2} \cup \{ (1, 0, 1) \} & (n = m).
    \end{cases}
\end{align*}

\begin{lem}\label{lem:Reduction}
    Let $n$, $m$ and $\ell$ be positive integers with $n \ge m$ and $\ell \ge 2$, and let $\theta = (r, s, t) \in \Theta$.
    Let $s_0 \equiv s \mod 2$ with $s_0 \in \{0,1\}$.
    Then
    \[
        G_\theta \cong G_{(r, s_0, (1 - r)t)}
    \]
    except the case $n = m$ and $\theta = (1, 0, 1)$.
    In the exceptional case,
    \[
        G_\theta \cong G_{(1, 1, 0)}.
    \]
    In particular, we have $G_{\scriptsizetextdice{\textnormal{?}}} \cong G_\theta$ if $\theta \in \Theta_{\scriptsizetextdice{\textnormal{?}}}$
    for $\textdice{\textnormal{?}} \in \{\dice{1}, \dotsc, \dice{6}\}$.
\end{lem}
\begin{proof}
    Assume that $n > m$ or $\theta \ne (1, 0, 1)$ so that $r(2 - t) 2^{n - m}$ is even.
    Note here that $n = m$ implies $s = 0$.

    Take an even number~$u$ such that $s + u = s_0 + r(2 - t)2^{n - m}$.
    Then $x^u$ is central and
    \begin{align*}
        (x^u y)^{2^m}
        &= x^{u 2^m} y^{2^m}
        =  x^{u 2^m} x^{s 2^m} z^{t 2^{\ell - 1}}
        =  x^{(s + u) 2^m} z^{t 2^{\ell - 1}} \\
        &= x^{s_0 2^m} x^{r(2 - t) 2^n} z^{t 2^{\ell - 1}}
        =  x^{s_0 2^m} z^{r^2(2 - t) 2^{\ell - 1}} z^{t 2^{\ell - 1}} \\
        &= x^{s_0 2^m} z^{r(2 - t) 2^{\ell - 1}} z^{t 2^{\ell - 1}}
        =  x^{s_0 2^m} z^{(1 - r)t 2^{\ell - 1} }.
    \end{align*}
    Hence $G_\theta = \langle x, x^u y, z \rangle$ enjoys the defining relations of $G_{(r, s_0, (1 - r)t)}$.

    If $n = m$ and $\theta = (1, 0, 1)$, then
    \(
        y^{2^m} = z^{2^{\ell - 1}} = x^{2^n} = x^{2^m}
    \)
    and $G_\theta = \langle x, y, z \rangle$ enjoys the defining relations of $G_{(1, 1, 0)}$.
\end{proof}

\begin{proof}[Proof of \cref{thm:Presentations}]
    It follows from \cref{lem:Cyclic,lem:Parametrization,lem:Reduction}.
\end{proof}

\section{Isomorphic modular group algebras}

This section is devoted to proving that groups $\DG{1}$ and $\DG{2}$ with $n > m > \ell \ge 2$
have isomorphic group algebras over an arbitrary field of characteristic~$2$, but are non-isomorphic groups.
In the next section, we will see that these are the only counterexamples to the modular isomorphism problem
within the class of two-generated finite $2$-groups with dihedral central quotient.

Arguably the most interesting part of this article is the following, which generalizes the negative solutions to the modular isomorphism problem presented in~\cite{GLMdR22}.

\begin{thm}\label{thm:IsomorphicAlgebras}
    Let $n \ge m > \ell \ge 2$ and $\F$ a field of characteristic~$2$.
    Then $\F \DG{1} \cong \F \DG{2}$.
\end{thm}
\begin{proof}
    Let $G = \DG{1}$ and $H = \DG{2} = \langle a, b, c \rangle$.
    We construct an algebra homomorphism from $\F G$ to $\F H$ and then show that it is bijective.
    Given a group homomorphism from $G$ to the unit group of $\F H$,
    we can extend it and obtain an algebra homomorphism from $\F G$ to $\F H$ by~\cite[Lemma 1.1.7]{Pas77}.
    To obtain such a group homomorphism, it suffices to show that the units
    \[
        x = a,\ y = b + a + 1,\ z = [y, x] \in \F H \smallsetminus \Delta(H)
    \]
    satisfy the defining relations for $G$ by~\cite[Proposition 4.3]{Joh97}.

    We first verify the commutator relations.
    Evidently $[y, x] = z$.
    As $x^2$ is central in $\F H$, the basic commutator identity $[y, x^2] = [y, x] [y, x]^x$ yields $[z, x] = z^{-2}$.
    Now use $ba = abc$ to see
    \[
        y^2 = b^2 + a^2 + ab + abc + 1.
    \]
    Observe that $(ab)^a = abc = (ab)^b$ and $(abc)^a = ab = (abc)^b$.
    Hence $\{ ab, abc \}$ is a conjugacy class of $H$ and the class sum $ab + abc$ is central in $\F H$.
    Thus so is $y^2$ which yields $[z, y] = z^{-2}$ as before.

    We proceed to the power relations.
    Clearly $x^{2^n} = 1$.
    As $ab$ commutes with $c$, it is also easy to raise $y$ to a power:
    \begin{align*}
        y^{2^m}
        &= (b^2 + a^2 + ab + abc + 1)^{2^{m - 1}} \\
        &= b^{2^m} + a^{2^m} + (ab)^{2^{m - 1}}(1 + c^{2^{m - 1}}) + 1.
    \end{align*}
    The first two terms vanish as $b^{2^m} = a^{2^m}$, and the third term vanishes as $c^{2^\ell} = 1$ and $m > \ell$.
    Thus $y^{2^m} = 1$.
    Finally, use $yx = xyz$ and $ba = abc$ to see
    \[
        xy(1 + z) = xy + yx = ab + ba = ab(1 + c).
    \]
    We raise both sides to the power of $2^\ell$.
    Since $ab$ commutes with $c$, the right hand side becomes
    $(ab)^{2^\ell}(1 + c^{2^\ell}) = 0$.
    On the other hand, $xy$ commutes with $z$ as
    \[
        [z, xy] = [z, y][z, x]^y = z^{-2} z^2 = 1.
    \]
    Hence the left hand side becomes
    $(xy)^{2^\ell}(1 + z^{2^\ell})$.
    As $xy$ is a unit, we obtain $z^{2^\ell} = 1$.
    Therefore all relations are satisfied and we obtain a suitable algebra homomorphism.

    Since $a = x$ and $b = y + x + 1$ belong to the image of our algebra homomorphism, it is surjective.
    As $G$ and $H$ have the same order, it is bijective.
\end{proof}

\begin{lem}\label{lem:G1vsG2}
    Let $n > m > \ell \ge 2$.
    Then $\DG{1} \not\cong \DG{2}$.
\end{lem}

\begin{proof}
    Let $G = \DG{1} = \langle x, y, z \rangle$ and $H = \DG{2} = \langle a, b, c \rangle$.
    The centralizers of the derived subgroups can distinguish these groups.
    First we calculate the exponent of $\CC_H(\D{H})$.

    Recall that the Frattini subgroup $\Frat(H)$ is generated by $a^2$, $b^2$ and $c$.
    As $a^2$ and $b^2$ are central we have $\Frat(H) \le \CC_H(\D{H})$.
    Moreover, as $ab \in \CC_H(\D{H})$ but $ab \notin \Frat(H)$
    we have that $\langle a^2, ab, b^2, c \rangle$ is an abelian maximal subgroup of $H$ and hence equals $\CC_H(\D{H})$.
    The orders of $a^2$, $b^2$ and $c$ are $2^{n - 1}$, $2^{n - 1}$ and $2^\ell$ respectively.
    From
    $(ab)^2 = a^2bcb = a^2b^2c^{-1}$ and $b^{2^m} = a^{2^m}$
    we obtain
    \[
        (ab)^{2^{n - 1}} = a^{2^{n - 1}}b^{2^{n - 1}}c^{-2^{n - 2}} = c^{-2^{n - 2}}.
    \]
    Since $n > m > \ell \ge 2$, we have $(ab)^{2^{n - 1}} = 1$ and the exponent of $\CC_H(\D{H})$ is $2^{n - 1}$.

    Similar arguments show that $\CC_G(\D{G})$ is abelian and generated by $x^2$, $xy$, $y^2$ and $z$.
    Then the exponent of $\CC_G(\D{G})$ is equal to $2^n$ which is the order of $xy$.
    Hence these groups are not isomorphic.
\end{proof}

The negative solutions to the modular isomorphism problem presented in~\cite{GLMdR22} correspond to the groups $\DG{1}$ and $\DG{2}$ for $n > m > \ell = 2$ in our notation.
In fact, the isomorphisms from $\DG{1} = \langle x, y, z \rangle$ and $\DG{2} = \langle a, b, c \rangle$ to the original groups
\begin{alignat*}{8}
    & \mathsf{G} &\ &= &\ \langle\, \mathsf{x}, \mathsf{y}, \mathsf{z} \mid
        \mathsf{x}^{2^n} &= 1, &\quad \mathsf{y}^{2^m} &= 1, &\quad \mathsf{z}^{2^\ell} &= 1, &\quad
            [\mathsf{y}, \mathsf{x}] &= \mathsf{z}, &\quad \mathsf{z}^\mathsf{x} &= \mathsf{z}^{-1}, &\quad \mathsf{z}^\mathsf{y} &= \mathsf{z}^{-1} \,\rangle, \\
    & \mathsf{H} &\ &= &\ \langle\, \mathsf{a}, \mathsf{b}, \mathsf{c} \mid
        \mathsf{a}^{2^n} &= 1, &\quad \mathsf{b}^{2^m} &= 1, &\quad \mathsf{c}^{2^\ell} &= 1, &
            [\mathsf{b}, \mathsf{a}] &= \mathsf{c}, &\quad \mathsf{c}^\mathsf{a} &= \mathsf{c}^{-1}, &\quad \mathsf{c}^\mathsf{b} &= \mathsf{c} \,\rangle
\end{alignat*}
are given by
\begin{alignat*}{4}
    \DG{1} &\to \mathsf{G}, &\quad x &\mapsto \mathsf{x}, &\quad y &\mapsto \mathsf{y},           &\quad z &\mapsto \mathsf{z}, \\
    \DG{2} &\to \mathsf{H}, &\quad a &\mapsto \mathsf{a}, &\quad b &\mapsto \mathsf{a}\mathsf{b}, &\quad c &\mapsto \mathsf{c}.
\end{alignat*}

\begin{rmk}
    For a finite $p$-group $G$ with cyclic derived subgroup $\D{G}$,
    Garc\'{\i}a-Lucas, del R\'{\i}o and Stanojkovski~\cite[Theorem A]{GLdRS23} prove that
    the exponent of the centralizer $\CC_G(\D{G})$ is an invariant of its modular group algebra, provided that the prime~$p$ is odd.
    This shows that the above pairs of groups, including the original ones presented in~\cite{GLMdR22}, are distinctive for $p = 2$.
\end{rmk}

The groups $\DG{1}$ and $\DG{2}$ are in fact isomorphic if $n = m$.
We summarize explicitly the negative solutions to the modular isomorphism problem that we obtain by the previous theorem and lemma.

\begin{cor}\label{cor:Counterexamples}
    Let $n > m > \ell \ge 2$.
    Then the groups
    \begin{align*}
        \langle\, x, y, z \mid x^{2^n} &= 1, & y^{2^m} &= 1,       & z^{2^\ell} &= 1, & [y, x] &= z, & [z, x] &= z^{-2}, & [z, y] &= z^{-2} \,\rangle, \\
        \langle\, a, b, c \mid a^{2^n} &= 1, & b^{2^m} &= a^{2^m}, & c^{2^\ell} &= 1, & [b, a] &= c, & [c, a] &= c^{-2}, & [c, b] &= c^{-2} \,\rangle
    \end{align*}
    are not isomorphic, but have isomorphic group algebras over an arbitrary field of characteristic~$2$.
\end{cor}

\begin{rmk}\label{rmk:OddPrimes}
    Given the simple proof of the negative solutions to the modular isomorphism problem for $p = 2$,
    it is natural to ask whether this strategy might be imitated for groups of odd order.
    While we currently have no clue,
    at least the isomorphism we exhibit and the proof we use can not be imitated directly, as we will quickly show.
    Analyzing what facilitates the proof so much,
    one important thing is that no actual commutator between any non-trivial units in the group algebra must be computed:
    the action of the generators of the group on the derived subgroup follows from the fact that their squares are central,
    and the correct structure of the whole derived subgroup essentially from its cyclicity.
    To imitate the proof, it would be necessary to find a finite non-abelian $p$-group $G$ generated by two elements,
    say $x$ and $y$, such that $x^p$ and $y^p$ are central and $\D{G}$ is cyclic.
    We show that in this case $G/\Z(G)$ is elementary abelian of rank two.
    For such groups,
    a positive answer to the modular isomorphism problem over an arbitrary field is obtained by Drensky~\cite{Dre89}.

    So assume~$p$ is odd, $G = \langle x, y\rangle$ is a finite non-abelian $p$-group such that $x^p, y^p$ are central in $G$ and $\D{G} = \langle z \rangle$ is cyclic.
    We assume $z = [y, x]$ and let $p^\ell$ be the order of $z$.
    Then $G$ is regular by~\cite[III, Satz 10.2 (c)]{Hup67}.
    Hence $1 = [z, x^p] = [z, x]^p$ by~\cite[III, Satz 10.6 (b)]{Hup67}.
    As $\Omega(\D{G}) = \langle z^{p^{\ell - 1}} \rangle$, there exists an $0 \le r \le p - 1$ such that $[z, x] = z^{r p^{\ell - 1}}$.
    This implies $z^{x^i}  = z^{1 + ir p^{\ell - 1}}$.
    We conclude by the standard commutator identity that $[y, x^p] = [y, x] [y, x]^x \dotsb [y, x]^{x^{p - 1}} = z^k$, where
    \[
        k = \sum_{i = 0}^{p - 1} (1 + irp^{\ell - 1}) = p + \tfrac{1}{2}r(p - 1)p^\ell \equiv p \mod p^\ell.
    \]
    As $[y, x^p] = 1$, we conclude that $\ell = 1$.
    Hence $x$ centralizes $z$, and so does $y$.
    This implies that $z$ is central.
    As $x^p$ and $y^p$ are also central, we get $\Frat(G) = \Z(G)$, which shows that $G/\Z(G)$ is indeed elementary abelian of rank two.
    
    We add one bibliography remark for readers interested in generalizing the groups from \cref{thm:IsomorphicAlgebras} to odd primes:
    in several generalizations we have tried and which were proposed to us the invariants from \cite{Bag99} turned out to solve the problem in the positive.
\end{rmk}

For the choice of our isomorphisms in \cref{thm:IsomorphicAlgebras}, see also \cref{rmk:Assignment}.

\section{Non-isomorphic modular group algebras}\label{sec:PositiveSolutions}
The rest of this article is devoted to showing that the groups $\DG{1}$ and $\DG{2}$ with $n > m > \ell \ge 2$ are indeed
the only negative solutions within our class, thereby completing the proof of \cref{main:A}.
Throughout this section, we fix a field $\F$ of characteristic~$p$ and assume $p = 2$ unless otherwise stated.

We first apply known group-theoretical invariants
and then utilize a well-known argument on power maps to distinguish algebras.
However, this turns out to be insufficient, and in the last part we apply a procedure that could be called a ``group base approximation''.
This could also be used for other classes of groups with some adjustments.

For the proof, we examine all the pairs of groups with $n \ge m$ and $\ell \ge 2$.
Basically, we divide it into five cases.
\begin{enumerate}
    \item $n > m > \ell$.
    \item $n > m \ge 2$ and $m \le \ell$.
    \item $n > m = 1$.
    \item $n = m \ge 2$.
    \item $n = m = 1$.
\end{enumerate}

We observe first that the coclass of the groups is $n + m - 1$, and the last case $n = m = 1$ corresponds to the finite $2$-groups of maximal class
for which a positive answer to the modular isomorphism problem has been known for decades~\cite{Car77, Bag92}.
Therefore we may assume $n \ge 2$.

The case $n = m \ge 2$ can be dealt with using a classic invariant, the center of a group, except the pair that corresponds to $\dice{5}\dice{6}$.
We will use a group base approximation for this case (\cref{lem:NeqMltLG5vsG6,lem:NeqMeqLG5vsG6,lem:MgtLG5vsG6}).
What will be used to distinguish modular group algebras for the rest of the cases is summarized in \cref{tbl:Summary}.
\begin{table}[htbp]
    \centering
    \begin{tabular}{p{1cm}p{2.5cm}p{2.5cm}p{2.5cm}}
        \toprule
         & $n > m > \ell$ & $n > m \ge 2$ and $m \le \ell$ & $n > m = 1$ \\
        \midrule
        $\dice{1}\dice{2}$ & X & P & P \\
        $\dice{1}\dice{3}$ & C & C & Q \\
        $\dice{1}\dice{4}$ & C & C & P \\
        $\dice{1}\dice{5}$ & C & C & C \\
        $\dice{1}\dice{6}$ & C & C & C \\
        $\dice{2}\dice{3}$ & C & C & P \\
        $\dice{2}\dice{4}$ & C & C & K \\
        $\dice{2}\dice{5}$ & C & C & C \\
        $\dice{2}\dice{6}$ & C & C & C \\
        $\dice{3}\dice{4}$ & A & P & P \\
        $\dice{3}\dice{5}$ & C & C & C \\
        $\dice{3}\dice{6}$ & C & C & C \\
        $\dice{4}\dice{5}$ & C & C & C \\
        $\dice{4}\dice{6}$ & C & C & C \\
        $\dice{5}\dice{6}$ & A & P & P \\
        \bottomrule \\
    \end{tabular}
    \caption{
        Summary of how modular group algebras will be distinguished for $n > m$.
        (X) counterexample           [\cref{thm:IsomorphicAlgebras}],
        (C) center of group          [\cref{lem:Center}],
        (P) kernel size of power map [\cref{lem:KernelSizes}],
        (Q) Quillen's theorem        [\cref{prop:Quillen}],
        (K) K\"ulshammer's theorem   [\cref{prop:Kuelshammer}],
        (A) group base approximation [\cref{subsec:GroupBaseApprox}].
    }
    \label{tbl:Summary}
\end{table}

\subsection{Group-theoretical invariants}

The first invariant we will use is the isomorphism type of the center of a group.
This is a well-known invariant, and a proof can be found in~\cite[III, Theorem 6.6]{Seh78}.

\begin{lem}\label{lem:Center}
    Let $G = \langle x, y, z \rangle$ be one of the groups $\DG{1}$, \dots, $\DG{6}$.
    Then the center of $G$ is generated by $x^2$, $y^2$ and $z^{2^{\ell - 1}}$.
    Consequently\textup{:}
    \begin{alignat*}{4}
        \Z(\DG{1}) &= \langle x^2 \rangle \times \langle y^2 \rangle        \times \langle z^{2^{\ell - 1}} \rangle &\quad &\cong \Cyc{2^{n - 1}} \times \Cyc{2^{m - 1}} \times \Cyc{2}; \\
        \Z(\DG{2}) &= \langle x^2 \rangle \times \langle x^{-2} y^2 \rangle \times \langle z^{2^{\ell - 1}} \rangle &\quad &\cong \Cyc{2^{n - 1}} \times \Cyc{2^{m - 1}} \times \Cyc{2}; \\
        \Z(\DG{3}) &= \langle x^2 \rangle \times \langle y^2 \rangle                                                &\quad &\cong \Cyc{2^{n - 1}} \times \Cyc{2^m}; \\
        \Z(\DG{4}) &= \langle x^2 \rangle \times \langle x^{-2} y^2 \rangle                                         &\quad &\cong \Cyc{2^{n - 1}} \times \Cyc{2^m}; \\
        \Z(\DG{5}) &= \langle x^2 \rangle \times \langle y^2 \rangle                                                &\quad &\cong \Cyc{2^n} \times \Cyc{2^{m - 1}}; \\
        \Z(\DG{6}) &= \langle x^2 \rangle \times \langle x^{-2} y^2 \rangle                                         &\quad &\cong \Cyc{2^n} \times \Cyc{2^{m - 1}}.
    \end{alignat*}
\end{lem}
\begin{proof}
    From the relations, we directly get $x^2, y^2 \in \Z(G)$.
    As $G/\langle x^2, y^2 \rangle$ is dihedral with the center generated by the image of $z^{2^{\ell - 1}}$,
    it follows that $\Z(G) = \langle x^2, y^2, z^{2^{\ell - 1}} \rangle$.
    The concrete descriptions of the centers are then easy to read off from the relations.
\end{proof}

\begin{prop}\label{prop:GTInvariantsNeqM}
    Let $n = m$.
    If $G$ and $H$ are one of the groups $\DG{1}$, \dots, $\DG{6}$ with $\F G \cong \F H$ but $G \not\cong H$,
    then $G$ and $H$ are isomorphic to the groups $\DG{5}$ and $\DG{6}$ with $n = m \ge 2$.
\end{prop}
\begin{proof}
    We observe first that the coclass of the groups is $n + m - 1$ in general.
    Hence the case $n = m = 1$ corresponds to the finite $2$-groups of maximal class for which
    a positive answer to the modular isomorphism problem is obtained by
    Carlson~\cite[p. 434]{Car77} and Bag\'{\i}nski~\cite{Bag92}.

    Thus we may assume that $n = m \ge 2$.
    Recall that $\DG{1} \cong \DG{2}$ and $\DG{3} \cong \DG{4} \cong \DG{5}$ for $n = m$.
    Since
        $\Z(\DG{1}) \cong \Cyc{2^{n - 1}} \times \Cyc{2^{m - 1}} \times \Cyc{2}$
    and
        $\Z(\DG{5}) \cong \Z(\DG{6}) \cong \Cyc{2^n} \times \Cyc{2^{m - 1}}$
    from \cref{lem:Center}, the assertion follows.
\end{proof}

We will see in \cref{lem:NeqMltLG5vsG6,lem:NeqMeqLG5vsG6,lem:MgtLG5vsG6} that
the remaining cases cannot happen and hence $n > m$.

\begin{prop}\label{prop:Kuelshammer}
    Let $n > m = 1$. Then $\F \DG{2} \not\cong \F\DG{4}$.
\end{prop}
\begin{proof}
    We will calculate the number of conjugacy classes that consist of squares for each group.
    This is an invariant of the group algebra due to the work of K\"ulshammer~\cite[Section 1]{Kue82},
    cf.~\cite[Section 2.2]{HS06} 
    for a short proof.
    To see that these numbers are different for $\DG{2}$ and $\DG{4}$, consider a generic element
    $x^r y^s z^t$ with $0 \le r \le 2^n - 1$, $0 \le s \le 1$ and $0 \le t \le 2^\ell - 1$.
    Then, independently of the group,
    \[
        (x^r y^s z^t)^2 = \begin{cases}
            x^{2r} z^{2t}         & (r \equiv 0 \mod 2, \ s = 0) \\
            x^{2r}                & (r \equiv 1 \mod 2, \ s = 0) \\
            x^{2r} y^2 z^{2t - 1} & (r \equiv 1 \mod 2, \ s = 1) \\
            x^{2r} y^2            & (r \equiv 0 \mod 2, \ s = 1).
        \end{cases}
    \]
    Note that if a square is not central,
    then its centralizer is given by the maximal subgroup $\CC_G(\D{G}) = \langle x^2, y^2, z, xy \rangle$,
    the centralizer of $\D{G}$.
    So each conjugacy class containing squares has either one or two elements;
    this can be seen from the generic form of an element.
    The case $r \equiv 0 \mod 2$ and $s = 0$ covers all elements of shape
    $x^u z^v$ with $u \equiv 0 \mod 4$ and $v \equiv 0 \mod 2$.
    These elements are formally reading the same in $\DG{2}$ and $\DG{4}$ and they also give the same number of classes,
    as the centrality of an element of this shape does not depend on which group we choose.
    In the case $ r \equiv 1 \mod 2$ and $s = 0$, we get elements of shape $x^u$ with $u \equiv 2 \mod 4$.
    Again this is the same for $\DG{2}$ and $\DG{4}$.
    Next we consider the case $r \equiv 1 \mod 2$ and $s = 1$.
    In $\DG{2}$, we have $x^{2r} y^2 z^{2t - 1} = x^{2r + 2} z^{2t - 1}$,
    so these are the elements of shape $x^u z^v$ with $u \equiv 0 \mod 4$ and $v \equiv 1 \mod 2$.
    In $\DG{4}$, we have $x^{2r} y^2 z^{2t - 1} = x^{2r + 2} z^{2t - 1 + 2^{\ell - 1}}$, which gives the same kind of elements.

    Finally, we consider the case where there is a difference: $r \equiv 0 \mod 2$ and $s = 1$.
    In $\DG{2}$, we have $x^{2r} y^2 = x^{2r + 2}$,
    which are elements of shape $x^u$ with $u \equiv 2 \mod 4$.
    These elements are not new as they already appeared in the second case (i.e. $r \equiv 1 \mod 2$ and $s = 0$).
    While in $\DG{4}$, we have $x^{2r} y^2 = x^{2r + 2} z^{2^{\ell - 1}}$,
    which are elements of shape $x^u z^{2^{\ell - 1}}$ for $u \equiv 2 \mod 4$.
    These elements are new, so in particular there are $2^{n - 2}$ more conjugacy classes of squares in $\DG{4}$ than in $\DG{2}$.
\end{proof}

\begin{prop}\label{prop:Quillen}
    Let $n > m = 1$. Then $\F \DG{1} \not\cong \F\DG{3}$.
\end{prop}
\begin{proof}
    Note that $\DG{1}$ contains an elementary abelian subgroup of rank three,
    namely $\langle x^{2^{n - 1}}, y, z^{2^{\ell - 1}} \rangle$.
    By contrast, $\DG{3}$ does not since the only involutions are $x^{2^{n - 1}}, z^{2^{\ell - 1}}$ and $x^{2^{n - 1}}z^{2^{\ell - 1}}$.
    As the maximal rank of an elementary abelian subgroup is known to be an invariant of the group algebra
    due to the work of Quillen~\cite[Theorem 6.28]{San85},
    we conclude that $\F \DG{1} \not\cong \F \DG{3}$.
\end{proof}

\begin{prop}\label{prop:GTInvariantsNgtM}
    Let $n > m$.
    If $G$ and $H$ are one of the groups $\DG{1}$, \dots, $\DG{6}$ with $\F G \cong \F H$ but $G \not\cong H$,
    then $G$ and $H$ are isomorphic to one of the following groups.
    \begin{enumerate}
        \item $\DG{1}$ and $\DG{2}$.
        \item $\DG{3}$ and $\DG{4}$.
        \item $\DG{5}$ and $\DG{6}$.
    \end{enumerate}
\end{prop}
\begin{proof}
    This follows from \cref{lem:Center,prop:Kuelshammer,prop:Quillen}.
\end{proof}

We note that computer experiments for the remaining cases in \cref{prop:GTInvariantsNeqM,prop:GTInvariantsNgtM} show that
the group-theoretical invariants contained in a \textsf{GAP} package \mbox{\textsf{ModIsomExt}} \cite{MM20, MM22}
do not give new information on whether the algebras in question are isomorphic.

\subsection{Kernel sizes}
We will write $\Delta$ for the augmentation ideal $\Delta(G)$ of $\F G$ for brevity.
The next argument we will apply is well known and consists in computing the sizes of the kernels of certain maps.
It can be traced back to an idea of Brauer~\cite[Supplement \S 9]{Bra63}
and was applied in practice for the first time by Passman~\cite{Pas65}.
We define the standard $2^m$-power map
\begin{equation}\label{eq:Phi}
    \varphi\colon \Delta/\Delta^2 \to \Delta^{2^m}/\Delta^{1 + 2^m}.
\end{equation}
The kernel of a power map $\varphi$ is defined to be
\begin{equation}
    K(G) = \set{ R + \Delta^2 \in \Delta/\Delta^2 \given \varphi(R + \Delta^2) = 0 + \Delta^{1 + 2^m} }.
\end{equation}
Note that this kernel is not an ideal of $\Delta$ in general, but still a well-defined set stable under automorphisms of the algebra.

A basis for $\Delta$ makes concrete calculations of the kernels feasible,
and we use what is called a Jennings basis, which behaves well with a filtration $\Delta \supseteq \Delta^2 \supseteq \dotsb$.
We will work with a basis that looks formally the same in each case.
We assume some familiarity with Jennings' theory for the proof of the next lemma.
If the reader is willing to accept it, then one can safely skip its proof because we will not use Jennings' theory explicitly elsewhere.

We write capital letters for elements of the augmentation ideal corresponding to group elements, e.g. $X = x + 1 \in \Delta$ for $x \in G$.
\begin{lem}\label{lem:BasisDelta}
    Let $G = \langle x, y, z \rangle$ be one of the groups $\DG{1}$, \dots, $\DG{6}$ and $k$ a non-negative integer.
    Set $w = z^{2^{\ell - 1}}$ and
    \[
        q = \begin{cases}
            2^\ell              & (\text{$G = \DG{1}$ or $G = \DG{2}$}) \\
            2^{\max\{m, \ell\}} & (\text{$G = \DG{3}$ or $G = \DG{4}$}) \\
            2^{\max\{n, \ell\}} & (\text{$G = \DG{5}$ or $G = \DG{6}$}).
        \end{cases}
    \]
    Then
    \begin{equation}\label{eq:BasisDelta}
        \FD_k = \left\{\, X^rY^sZ^tW^u \mathrel{}\middle|\mathrel{}
            \begin{gathered}
                0 \le r \le 2^n - 1,\ 0 \le s \le 2^m - 1,\\
                0 \le t \le 2^{\ell - 1} - 1,\ 0 \le u \le 1,\\
                r + s + 2t + qu \ge k
            \end{gathered}
        \,\right\}
    \end{equation}
    is a basis of $\Delta^k$.
    In particular, the image of $\FD_k \smallsetminus \FD_{k + 1}$ under the natural projection $\Delta^k \to \Delta^k/\Delta^{k + 1}$ is a basis of $\Delta^k/\Delta^{k + 1}$.
\end{lem}
\begin{proof}
    Let $M_k = G \cap (1 + \Delta^k)$, the $k$th dimension subgroup of $G$, for $k \ge 1$.
    Observe first that $\D{G} = \langle z \rangle$ and $\Frat(G) = \langle x^2, y^2, z \rangle$ are abelian.
    It follows from~\cite[Theorems 11.2 and 12.9]{DdSMS99} that $M_1 = \langle x, y \rangle$ and
    \begin{equation}
        2^{e - 1} < k \le 2^e \implies M_k = \Agemo_e(G)\Agemo_{e - 1}(\D{G}) = \langle x^{2^e}, y^{2^e}, z^{2^{e - 1}} \rangle
    \end{equation}
    for $e \ge 1$.
    In particular, we have
    \begin{alignat*}{3}
        0 &\le e \le n - 1    &\quad &\implies &\quad x^{2^e}       &\in M_{2^e} \smallsetminus M_{1 + 2^e}, \\
        0 &\le e \le m - 1    &\quad &\implies &\quad y^{2^e}       &\in M_{2^e} \smallsetminus M_{1 + 2^e}, \\
        1 &\le e \le \ell - 1 &\quad &\implies &\quad z^{2^{e - 1}} &\in M_{2^e} \smallsetminus M_{1 + 2^e}
    \end{alignat*}
    and by the relations of $G$,
    \[
        w \in M_{q} \smallsetminus M_{1 + q}.
    \]
    Since $x^2$ and $y^2$ are central, an element of the Jennings basis can be written in the form
    \begin{align*}
        &(x + 1)^{r_0} (y + 1)^{s_0} (x^2 + 1)^{r_1} (y^2 + 1)^{s_1} (z + 1)^{t_0} \dotsm (w + 1)^u \dotsm \\
        &\qquad= (x + 1)^{r_0} (x^2 + 1)^{r_1} \dotsm (y + 1)^{s_0} (y^2 + 1)^{s_1} \dotsm (z + 1)^{t_0} \dotsm (w + 1)^u \\
        &\qquad= (x + 1)^{r_0} (x + 1)^{r_1 2} \dotsm (y + 1)^{s_0} (y + 1)^{s_1 2} \dotsm (z + 1)^{t_0} \dotsm (w + 1)^u \\
        &\qquad= (x + 1)^{r_0 + \dotsm + r_{n - 1} 2^{n - 1}} (y + 1)^{s_0 + \dotsm + s_{m - 1} 2^{m - 1}} (z + 1)^{t_0 + \dotsm + t_{\ell - 2} 2^{\ell - 2}} (w + 1)^u \\
        &\qquad= (x + 1)^r (y + 1)^s (z + 1)^t (w + 1)^u \\
        &\qquad= X^r Y^s Z^t W^u
    \end{align*}
    for
    \[
        0 \le r_0, r_1, \dotsc, r_{n - 1}, s_0, s_1, \dotsc, s_{m - 1}, t_0, \dotsc, t_{\ell - 2}, u \le 1
    \]
    where
    $r = r_0 + \dotsb + r_{n - 1} 2^{n - 1}$,
    $s = s_0 + \dotsb + s_{m - 1} 2^{m - 1}$ and
    $t = t_0 + \dotsb + t_{\ell - 2} 2^{\ell - 2}$.
    Hence the assertion follows from Jennings' theorem~\cite[Theorem 3.2]{Jen41}.
\end{proof}

The \emph{weight} of an element of $\FD_k \smallsetminus \FD_{k + 1}$ is defined to be $k$.

To aid the reader in the following calculations, we explicitly state
the bases of the first few layers of the subsequent quotients of the augmentation ideal power series when $n \ge m \ge 2$:
\begin{align*}
    \Delta^1/\Delta^2 &: & \FD_1 \smallsetminus \FD_2 &= \{ X, Y \}, \\
    \Delta^2/\Delta^3 &: & \FD_2 \smallsetminus \FD_3 &= \{ X^2, XY, Y^2, Z \}, \\
    \Delta^3/\Delta^4 &: & \FD_3 \smallsetminus \FD_4 &= \{ X^3, X^2Y, XY^2, XZ, Y^3, YZ \}. \\
\end{align*}

We will collect some elementary relations between the elements of the basis,
which we will use without further mention:
\begin{align*}
    XY + YX &= Z + XZ + YZ + XYZ, \\
    XY + YX &\equiv Z \mod \Delta^3.
\end{align*}

Now a more concrete expression of the power map \eqref{eq:Phi} can be obtained.

\begin{lem}\label{lem:FormulaPowerMap}
    Let $G = \langle x, y, z \rangle$ be one of the groups $\DG{1}$, \dots, $\DG{6}$ and $\alpha, \beta \in \F$.
    Then
    \[
        \varphi(\alpha X + \beta Y + \Delta^2)
        = \alpha^{2^m} X^{2^m} + \beta^{2^m} Y^{2^m} + (\alpha\beta)^{2^{m - 1}}Z^{2^{m - 1}} + \Delta^{1 + 2^m}.
    \]
\end{lem}
\begin{proof}
    In general we have
    \begin{align*}
        (\alpha X + \beta Y)^2
        &= \alpha^2 X^2 + \beta^2 Y^2 + (\alpha \beta)(XY + YX) \\
        &\equiv \alpha^2 X^2 + \beta^2 Y^2 + (\alpha \beta)Z \mod \Delta^3.
    \end{align*}
    Since $X^2$ and $Y^2$ are central in $\F G$, we obtain
    \[
        \varphi(\alpha X + \beta Y + \Delta^2)
        = \alpha^{2^m} X^{2^m} + \beta^{2^m} Y^{2^m} + (\alpha\beta)^{2^{m - 1}}Z^{2^{m - 1}} + \Delta^{1 + 2^m}.
        \qedhere
    \]
\end{proof}

\begin{lem}\label{lem:KernelSizes}
    Assume $n > m$ and let $G$ be one of the groups $\DG{1}$, \dots, $\DG{6}$.
    Then the kernel $K(G)$ of the power map $\varphi$ has the following cardinality as shown in \cref{tbl:KernelSizes}.
    \begin{table}[htbp]
        \centering
        \begin{tabular}{cccc}
            \toprule
             & $m > \ell$ & $m = \ell$ & $m < \ell$ \\
            \midrule
            \dice{1} & $|\F|$ & $|\F|$ & $|\F|$ \\
            \dice{2} & $|\F|$ & $1$    & $1$    \\
            \dice{3} & $1$    & $1$    & $|\F|$ \\
            \dice{4} & $1$    & $|\F|$ & $1$    \\
            \dice{5} & $|\F|$ & $|\F|$ & $|\F|$ \\
            \dice{6} & $|\F|$ & $1$    & $1$    \\
            \bottomrule \\
        \end{tabular}
        \caption{
            Kernel sizes $|K(G)|$ for $n > m$.
        }
        \label{tbl:KernelSizes}
    \end{table}
\end{lem}
\begin{proof}
    Let $G = \langle x, y, z \rangle$.
    Every element of $\Delta/\Delta^2$ can be written in the form
    $\alpha X + \beta Y + \Delta^2$
    and the scalars $\alpha, \beta \in \F$ are unique by \cref{lem:BasisDelta}.
    To measure the size of the kernel, we write
    $\varphi(\alpha X + \beta Y + \Delta^2)$
    in \cref{lem:FormulaPowerMap} as a linear combination of the images of 
    $\FD_{2^m} \smallsetminus \FD_{1 + 2^m}$ in $\Delta^{2^m}/\Delta^{1+2^m}$ for each case.
    Note that as $n > m$, the differences between $\DG{1}$ and $\DG{5}$ as well as between $\DG{2}$ and $\DG{6}$
    cannot enter into the expression of
    $\varphi(\alpha X + \beta Y + \Delta^2)$,
    so we handle these two pairs of groups simultaneously.
    We also use the equation
    $Y^{2^m} = X^{2^m} + W + X^{2^m}W$
    for $G = \DG{4}$.

    For $m > \ell$, as $Z^{2^{m - 1}}$ vanishes,
    $\varphi(\alpha X + \beta Y + \Delta^2)$ is written as follows.
    \begin{align*}
        \dice{1}\dice{5} &: &                                \alpha^{2^m}X^{2^m} + \Delta^{1 + 2^m}; & \\
        \dice{2}\dice{6} &: &               (\alpha^{2^m} + \beta^{2^m}) X^{2^m} + \Delta^{1 + 2^m}; & \\
        \dice{3}         &: &                 \alpha^{2^m}X^{2^m} + \beta^{2^m}W + \Delta^{1 + 2^m}; & \\
        \dice{4}         &: & (\alpha^{2^m} + \beta^{2^m})X^{2^m} + \beta^{2^m}W + \Delta^{1 + 2^m}. &
    \end{align*}

    Similarly, for $m < \ell$, the following is obtained.
    \begin{align*}
        \dice{1}\dice{5} &: &                  \alpha^{2^m}X^{2^m} + (\alpha\beta)^{2^{m - 1}} Z^{2^{m - 1}} + \Delta^{1 + 2^m}; & \\
        \dice{2}\dice{6} &: & (\alpha^{2^m} + \beta^{2^m}) X^{2^m} + (\alpha\beta)^{2^{m - 1}} Z^{2^{m - 1}} + \Delta^{1 + 2^m}; & \\
        \dice{3}         &: &                  \alpha^{2^m}X^{2^m} + (\alpha\beta)^{2^{m - 1}} Z^{2^{m - 1}} + \Delta^{1 + 2^m}; & \\
        \dice{4}         &: &  (\alpha^{2^m} + \beta^{2^m})X^{2^m} + (\alpha\beta)^{2^{m - 1}} Z^{2^{m - 1}} + \Delta^{1 + 2^m}. &
    \end{align*}

    Finally, for $m = \ell$, the expressions for $\DG{1}$, $\DG{2}$, $\DG{5}$ and $\DG{6}$ are as in the case $m < \ell$
    and for the other two we have the following.
    \begin{align*}
        \dice{3} &: &                 \alpha^{2^m}X^{2^m} + (\beta^{2^m} + (\alpha\beta)^{2^{m - 1}}) W + \Delta^{1 + 2^m}; & \\
        \dice{4} &: & (\alpha^{2^m} + \beta^{2^m})X^{2^m} + (\beta^{2^m} + (\alpha\beta)^{2^{m - 1}}) W + \Delta^{1 + 2^m}. &
    \end{align*}

    All of the terms present are linearly independent by \cref{lem:BasisDelta}.
    Hence $\alpha X + \beta Y + \Delta^2$ belongs to the kernel $K(G)$
    if and only if all of the coefficients present are equal to zero.
    Since the Frobenius map on $\F$ is injective,
    the cardinalities for $K(G)$ now follow easily.
\end{proof}

We remark that in the case $n = m$ the kernel size can also give some useful information,
but only if one restricts the possibilities for $\F$.
As we prefer to work independently of the base field,
we do not include the corresponding calculations.

We summarize the progress we have made on the proof of \cref{main:A} in this section.
\begin{prop}\label{prop:RemainingCases}
    If $G$ and $H$ are one of the groups $\DG{1}$, \dots, $\DG{6}$ with $\F G \cong \F H$ but $G \not\cong H$,
    then $G$ and $H$ are isomorphic to one of the following groups.
    \begin{enumerate}
        \item $\DG{1}$ and $\DG{2}$ with $n > m > \ell$.
        \item $\DG{3}$ and $\DG{4}$ with $n > m > \ell$.
        \item $\DG{5}$ and $\DG{6}$ with $n > m > \ell$.
        \item $\DG{5}$ and $\DG{6}$ with $n = m \ge 2$.
    \end{enumerate}
\end{prop}
\begin{proof}
    By \cref{prop:GTInvariantsNeqM}, we may assume that $n > m$.
    It follows from \cref{prop:GTInvariantsNgtM} that $G$ and $H$ are isomorphic to one of the followings:
    $\DG{1}$ and $\DG{2}$, $\DG{3}$ and $\DG{4}$, $\DG{5}$ and $\DG{6}$.
    If $m = \ell$ or $m < \ell$,
    then the groups in each pair have different kernel sizes by \cref{lem:KernelSizes}, which leads to a contradiction.
    Hence we have $m > \ell$.
\end{proof}

Some more information can be obtained from canonical ideals.
The following proposition could be used to further reduce the cases,
though we will not use it in this sense, but rather differently later on.
It seems to have the potential to also be applicable to other classes of groups.

\begin{prop}\label{prop:AgemoCenter}
    Let $G$ be a finite $p$-group and $\F$ a field of characteristic~$p$.
    Assume that $\D{G}$ is abelian of exponent $p^\ell$.
    Then
    \[
        \F(\Agemo_r(\Z(G))) = \Agemo_r(\Z(\F G))
    \]
    for every integer $r \ge \ell$ and, in particular, $\Delta(\Agemo_r(\Z(G)))$ is canonical in $\F G$.
    Consequently, if $\F G \cong \F H$ then $\F[G/\Agemo_r(\Z(G))] \cong \F[H/\Agemo_r(\Z(H))]$.
\end{prop}

\begin{proof}
    Recall that $\Z(\F G)$ has a basis that is given by class sums
    and a decomposition $\Z(\F G) = \F \Z(G) \oplus (\Z(\F G) \cap [\F G, \F G])$.
    We will show that class sums of non-central elements vanish under a $p^r$-power.
    So let $g \in G$ be a non-central element with conjugacy class
    $\{ ga_1, \dotsc, ga_s \}$ for some $a_1, \dotsc, a_s \in \D{G}$.
    Note that $p$ divides $s$.
    The class sum of $g$ equals $g(a_1 + \dotsb + a_s)$.
    As this element is central in $\F G$, it commutes in particular with $g$,
    hence $(a_1 + \dotsb + a_s)$ commutes with $g$.
    As $\D{G}$ is abelian of exponent $p^\ell \le p^r$, we conclude that
    \[
        (g(a_1 + \dotsb + a_s))^{p^r} = g^{p^r}(a_1^{p^r} + \dotsb + a_s^{p^r}) =  g^{p^r} s = 0.
        \qedhere
    \]
\end{proof}

The next lemma is an easy observation which will be helpful in some calculations.

\begin{lem}\label{lem:XYCommIsCentral}
    Let $G = \langle x, y, z \rangle$ be one of the groups $\DG{1}$, \dots, $\DG{6}$.
    Then $[X, Y]$ is central in $\F G$.
\end{lem}
\begin{proof}
    We have $[X, Y] = xy + yx = xy + xyz$.
    As $(xy)^x = xyz = (xy)^y$ and $(xyz)^x = xy = (xyz)^y$,
    we conclude that $[X, Y]$ is a class sum and hence central in $\F G$.
\end{proof}

\subsection{Group base approximation}\label{subsec:GroupBaseApprox}

We will handle the rest of the cases by a procedure that could be called a ``group base approximation''.
We will first make a few general observations on group bases satisfying the relations of the normal form given in \cref{thm:Presentations},
and then see in each single case why this leads to a contradiction.

\subsubsection{General group bases}
Now we fix integers $n$, $m$ and $\ell$ with $n \ge m \ge 2$ and $\ell \ge 2$.
Fix a group $G$ from $\DG{1}$, \dots, $\DG{6}$ generated by $x$, $y$ and $z$ which satisfy the defining relations.
Also fix another group base $H$ in $\F G$ from $\DG{1}$, \dots, $\DG{6}$ generated by $a$, $b$ and $c$ which satisfy the defining relations.
Thus, throughout this section, we assume
\begin{equation}\label{eq:Common}
    \begin{gathered}
        n \ge m \ge 2, \ \ell \ge 2, \  
        G = G_{\scriptsizetextdice{\textnormal{?}}}, \ H = G_{\scriptsizetextdice{\textnormal{?`}}}, \ 
        \textdice{\textnormal{?}}, \textdice{\textnormal{?`}} \in  \{\dice{1}, \dotsc, \dice{6}\}, \\
        G = \langle x, y, z \rangle, \ a, b, c \in 1 + \Delta(G), \ H = \langle a, b, c \rangle, \
        \F G = \F H.
    \end{gathered}
\end{equation}
As before, we write $\Delta$ for $\Delta(G)$ and use the basis $\FD_k$ of $\Delta^k$ introduced in \cref{lem:BasisDelta}.
Recall that $w = z^{2^{\ell - 1}}$, so in this basis $W = Z^{2^{\ell - 1}}$, and the weight of $W$ is denoted by $q$;
the exact value of $q$ will be irrelevant, except the fact that $q \ge 4$ as $\ell \ge 2$.

Since $n \ge m \ge 2$, the squares $x^2$ and $y^2$ are not trivial and we have
\[
    \Delta = \F X \oplus \F Y \oplus \F XY \oplus \F Z \oplus \F X^2 \oplus \F Y^2 \oplus \Delta^3.
\]

From now on square brackets denote Lie commutators in $\Delta$,
i.e. $[U, V] = UV + VU$ for $U, V \in \Delta$.
We will frequently use without comment the commutator formula
\[
    [S, T] = (1 + S + T + ST)(1 + t^{-1}s^{-1}ts)
\]
which holds for all $s, t \in G$.
In particular, $[X, Y] = (1 + X + Y + XY)Z$.
Moreover, as $z^{-1}z^x, z^{-1}z^y \in \langle z^2 \rangle$, we have  $[X, Z], [Y, Z] \in Z^2\F G \subseteq \Delta^4$.

We first consider $A$ and $B$ modulo $\Delta^3$.
Write
\begin{equation}\label{eq:AB}
    \begin{aligned}
        A &= \alpha_A X + \beta_A Y + \gamma_A XY + \delta_A Z + \xi_A X^2 + \eta_A Y^2 + U_A, \\
        B &= \alpha_B X + \beta_B Y + \gamma_B XY + \delta_B Z + \xi_B X^2 + \eta_B Y^2 + U_B
    \end{aligned}
\end{equation}
where $\alpha_A, \alpha_B, \dotsc, \eta_A, \eta_B \in \F$ and $U_A, U_B \in \Delta^3$.
Note that the scalars are uniquely determined.
This notation will be fixed throughout this section.
Observe that $A + \Delta^2$ and $B + \Delta^2$ span
the two-dimensional vector space $\Delta/\Delta^2$, as $a$ and $b$ generate $H$.
Hence they must be linearly independent, and we have
\begin{align}\label{eq:Determinant}
    \alpha_A\beta_B + \alpha_B\beta_A \ne 0.
\end{align}

We first show that $XY$ does not appear as a summand in $A$ and $B$.

\begin{lem}\label{lem:Gammas}
    We have $\gamma_A = \gamma_B = 0$.
\end{lem}
\begin{proof}
    The proof is identical for $A$ and $B$, so we only show $\gamma_A = 0$.
    The idea is to use the fact that $A^2$ is central in $\F G$, as $a^2$ is central in $H$.

    Let $\Xi$ be the linear subspace of $\Delta$ defined by
    \[
        \Xi = \set{ R \in \Delta^2 \given \text{$[R, S] \in \Delta^5$ for all $S \in \Delta$} }.
    \]
    Note that $\Delta^4 \subseteq \Xi$ and $[X, Y] \in \Xi$ by \cref{lem:XYCommIsCentral}.
    We will consider $A^2$ modulo $\Xi$.
    We will use that
    \begin{align*}
        [X, Z], \ & [Y, Z] \in Z^2\F G \subseteq \Delta^4 \subseteq \Xi, \\
        [X, XY] &= X[X, Y] \equiv XZ \mod \Xi, \\
        [Y, XY] &= [X, Y]Y \equiv YZ \mod \Xi.
    \end{align*}

    From \eqref{eq:AB}, we get
    \begin{equation}\label{eq:A2Modulo}
        \begin{aligned}
            A^2 &\equiv \alpha_A\beta_A [X, Y] + \alpha_A\gamma_A [X, XY] + \beta_A\gamma_A[Y, XY] \\
                &\equiv \alpha_A\gamma_A X[X, Y] + \beta_A\gamma_A[Y, X]Y \\
                &\equiv \alpha_A \gamma_A XZ + \beta_A \gamma_A YZ \mod \Xi
        \end{aligned}
    \end{equation}
    where $[X, Y] \in \Xi$ is used to go from the first to the second line.
    Computing commutators modulo $\Delta^5$ gives
    \begin{align*}
        [XZ, Y] &\equiv [YZ, X] \equiv Z^2 \mod \Delta^5, \\
        [XZ, X] &\equiv [YZ, Y] \equiv 0   \mod \Delta^5.
    \end{align*}
    So modulo $\Delta^5$ computing commutators with $A^2$ using \eqref{eq:A2Modulo} and the definition of $\Xi$ we obtain
    \begin{align*}
        [A^2, X] &\equiv  \beta_A\gamma_A Z^2 \mod \Delta^5, \\
        [A^2, Y] &\equiv \alpha_A\gamma_A Z^2 \mod \Delta^5.
    \end{align*}
    As $A^2$ is central in $\F G$, we conclude that $\alpha_A \gamma_A = \beta_A \gamma_A = 0$.
    But as $\alpha_A \ne 0$ or $\beta_A \ne 0$ by $A \not\in \Delta^2$, this implies $\gamma_A = 0$.
\end{proof}

The next lemma allows us to make conclusions about scalars from a relation between $A^{2^n}$ and $B^{2^n}$.

\begin{lem}\label{lem:A2NB2N}
    We have
    \begin{align*}
        A^{2^n} &\equiv \alpha_A^{2^n} X^{2^n} + \beta_A^{2^n} Y^{2^n} \mod \Z(\F G) \cap [\F G, \F G], \\
        B^{2^n} &\equiv \alpha_B^{2^n} X^{2^n} + \beta_B^{2^n} Y^{2^n} \mod \Z(\F G) \cap [\F G, \F G].
    \end{align*}
\end{lem}
\begin{proof}
    We first observe that $\Z(G)$ is generated by $x^2$, $y^2$ and $w$.
    Thus
    \[
        \F\Z(G) \cap \Delta
        = \sum_{\substack{r, s, t\ge 0\\ 2r + 2s + qt \ge 1}} \F X^{2r}Y^{2s}W^t \subseteq \F X^2 + \F Y^2 + \Delta^4.
    \]

    We will prove the claims only for $A$ as the arguments for $B$ are identical.
    Recall that \eqref{eq:AB} and \cref{lem:Gammas} imply
    \[
        A = \alpha_A X + \beta_A Y + \delta_A Z + \xi_A X^2 + \eta_A Y^2 + U_A.
    \]
    So, by $[X, Y] \equiv Z \mod \Delta^3$, it yields
    \begin{equation}\label{eq:A2Delta3}
        A^2 \equiv \alpha_A^2 X^2 + \beta_A^2 Y^2 + \alpha_A \beta_A Z \mod \Delta^3.
    \end{equation}
    Using the decomposition $\Z(\F G) = \F\Z(G) \oplus (\Z(\F G) \cap [\F G, \F G])$,
    we get
    \begin{equation}\label{eq:A2RS}
        A^2 = \alpha_A^2 X^2 + \beta_A^2 Y^2 + R_A + S_A
    \end{equation}
    for some $R_A \in \F\Z(G) \cap \Delta$ and $S_A \in \Z(\F G) \cap [\F G, \F G]$.

    Now $S_A$ can be written as a linear combination of class sums of non-central elements.
    It follows from $\D{G} = \langle z \rangle$ that
    \[
        S_A \in \F Z + \Delta^3.
    \]
    Indeed, let $g = x^iy^jz^k$ be a representative of a non-central conjugacy class of $G$ with $i, j, k \ge 0$.
    Since the conjugacy class of $g$ can be written as
    \[
        \{ gz^{e_1}, \dotsc, gz^{e_t} \}
    \]
    for some non-trivial power $t$ of 2 and some $e_1, \dotsc, e_t \ge 0$,
    the class sum of the conjugacy class can be written as follows.
    \begin{align*}
        gz^{e_1} + \dotsb + gz^{e_t}
        &= x^i y^j z^k(z^{e_1} + \dotsb + z^{e_t}) \\
        &= (1 + X)^i (1 + Y)^j (1 + Z)^k ((1 + Z)^{e_1} + \dotsb + (1 + Z)^{e_t}) \\
        &\equiv (1 + X)^i (1 + Y)^j (1 + Z)^k (e_1 + \dotsb + e_t) Z \\
        &\equiv (e_1 + \dotsb + e_t)Z \mod \Delta^3.
    \end{align*}
    Hence $S_A$ belongs to the desired vector space.
    
    Since $X^2 + \Delta^3$, $Y^2 + \Delta^3$, $Z + \Delta^3$ are linearly independent in $\Delta^2/\Delta^3$,
    it follows, reformulating \eqref{eq:A2Delta3} and \eqref{eq:A2RS} as
    \[
        R_A + S_A \equiv \alpha_A \beta_A Z \mod \Delta^3,
    \]
    that $R_A \in \Delta^4$.
    Since
    $X^{2^{n + 1}} = 0$,
    $X^{2^n} Y^{2^n} = 0$,
    $Y^{2^{n + 1}} = 0$ and
    $W^2 = 0$
    by the relations of $G$,
    we get $R_A^{2^{n - 1}} = 0$.
    Raising \eqref{eq:A2RS} to a power hence yields
    \[
        A^{2^n} = \alpha_A^{2^n} X^{2^n} + \beta_A^{2^n} Y^{2^n} + S_A^{2^{n - 1}},
    \]
    and the claim follows.
\end{proof}

The element $C = 1 + c$ will play an important role and we first approximate it modulo $\Delta^4$.

\begin{lem}\label{lem:CModulo}
    The following congruence holds\textup{:}
    \[
        C \equiv \lambda Z + \mu XZ + \nu YZ \mod \Delta^4,
    \]
    where
    $\lambda = \alpha_A\beta_B + \alpha_B \beta_A$,
    $\mu = \lambda(1 + \alpha_A + \alpha_B)$ and
    $\nu = \lambda(1 + \beta_A + \beta_B)$.
\end{lem}
\begin{proof}
    From the definitions, we have
    \[
        C = 1 + c = 1 + b^{-1}a^{-1}ba = 1 + (1 + B)^{-1}(1 + A)^{-1}(1 + B)(1 + A).
    \]
    For $R \in \Delta$, one has $(1 + R)^{-1} = 1 + R + R^2 + \dotsb + R^k$ for some $k$ such that $R^k = 0$.
    Hence
    \begin{align*}
        C
        &\equiv 1 + (1 + B + B^2 + B^3)(1 + A + A^2 + A^3)(1 + B)(1 + A) \\
        &\equiv  [A, B] + A[A, B] + B[A, B] \\
        &= (1 + A + B)[A, B] \mod \Delta^4
    \end{align*}
    by an explicit multiplication, with cancellation afterwards as we may in characteristic~$2$.
    Then, by $[X, Z], [Y, Z] \in \Delta^4$, $U_A, U_B \in \Delta^3$ and $\gamma_A = \gamma_B = 0$ (i.e. \cref{lem:Gammas}), from \eqref{eq:AB} we obtain
    \[
        C \equiv (1 + \alpha_AX + \beta_AY + \alpha_BX + \beta_BY)[\alpha_AX + \beta_AY, \alpha_BX + \beta_BY] \mod \Delta^4.
    \]
    Since $[X, Y] = (1 + X + Y + XY)Z$, straightforward calculations yield that
    \begin{align*}
        C
        &\equiv (1 + (\alpha_A + \alpha_B)X + (\beta_A + \beta_B)Y)(\alpha_A \beta_B + \alpha_B \beta_A)(Z + XZ + YZ) \\
        &\equiv (\alpha_A\beta_B + \alpha_B \beta_A)(Z + (1 + \alpha_A + \alpha_B)XZ + (1 + \beta_A + \beta_B)YZ) \mod \Delta^4.
        \qedhere
    \end{align*}
\end{proof}

We now define a variation of a power map that will turn out to be very useful to us.
First set $\Gamma = \Delta(\D{G})\F G = [\F G, \F G]\F G$.
Note that $\Gamma = Z\F G$, and thus the subset
\begin{equation}\label{eq:BasisGamma}
    \FC = \left\{\, X^r Y^s Z^t W^u \mathrel{}\middle|\mathrel{}
        \begin{gathered}
            0 \le r \le 2^n - 1,\ 0 \le s \le 2^m - 1,\\ 0 \le t \le 2^{\ell - 1} - 1,\ 0 \le u \le 1,\\ 2t + qu \ge 1
        \end{gathered}
    \,\right\}
\end{equation}
of the basis of $\F G$ defined in \cref{lem:BasisDelta} is a basis of $\Gamma$.

Recall that $q$ denotes the weight of $W$.
Set $d = 1 + 2^{\ell - 1} + q$, and define the $2^{\ell - 1}$-power map
\begin{equation}
    \psi\colon \Gamma/\Gamma^2 \to \Gamma^{2^{\ell - 1}}/\left(\Gamma^{1 + 2^{\ell - 1}} + \Delta^d\right).
\end{equation}
As will become clear, the reason for the definition of $d$ is that
$W$, $X^{2^{\ell - 1}}W$ and $Y^{2^{\ell - 1}}W$ are non-zero modulo $\Gamma^{1 + 2^{\ell - 1}} + \Delta^d$,
while other elements in $\FC$ that span the image of $\psi$ are zero.

\begin{lem}\label{lem:Psi}
    For $T \in \Delta^2$ and $\lambda, \mu, \nu \in \F$, we have
    \begin{align*}
        &\psi\left(\lambda Z + \mu XZ + \nu YZ + TZ + \Gamma^2\right)	 \\
        &\qquad= \lambda^{2^{\ell - 1}}W + \mu^{2^{\ell - 1}}X^{2^{\ell - 1}}W + \nu^{2^{\ell - 1}}Y^{2^{\ell - 1}}W
        + \left(\Gamma^{1 + 2^{\ell - 1}} + \Delta^d\right).
    \end{align*}
\end{lem}

\begin{proof}
    To prove the lemma, we factor $\psi$ as a composition of maps.
    Define the square-map
    \[
        \psi_i \colon \Gamma^{2^i}/\Gamma^{1 + 2^i} \to \Gamma^{2^{i + 1}}/\Gamma^{1 + 2^{i + 1}}
    \]
    for $i \ge 0$ and the natural projection
    \[
        \pi\colon \Gamma^{2^{\ell - 1}}/\Gamma^{1 + 2^{\ell - 1}} \to \Gamma^{2^{\ell - 1}}/\left(\Gamma^{1 + 2^{\ell - 1}} + \Delta^d\right).
    \]
    Then $\psi = \pi \circ \psi_{\ell - 2} \circ \dotsb \circ \psi_1 \circ \psi_0$.
    Note that this expression is well defined as $\ell \ge 2$.

    \textbf{Claim 1:} For $g \in G$ and $i \ge 0$, one has $[g, Z^{2^i}] \in \Gamma^{2^{i + 1}}$.

    It suffices to show that $1 + g^{-1} Z^{2^i} g \equiv 1 + Z^{2^i} \mod \Gamma^{2^{i + 1}}$.
    From the general formula $(1 + R)^{-1} = 1 + R + R^2 + \dotsb$ for $R \in \Delta$, we get
    \[
        1 + Z^{-2^i} = (1 + Z^{2^i})^{-1} = 1 + Z^{2^i} + (Z^{2^i})^2 + \dotsb \equiv 1 + Z^{2^i} \mod \Gamma^{2^{i + 1}}.
    \]
    Note that $\{ z^{2^i}, z^{-2^i} \}$ is a conjugacy class of $G$.
    Thus $1 + g^{-1} Z^{2^i} g$ is equal to $1 + Z^{2^i}$ or $1 + Z^{-2^i}$ and the claim holds.

    \textbf{Claim 2:} For $R, S \in \F G$, one has
    \[
        \psi_i\left(RZ^{2^i} + SZ^{2^i} + \Gamma^{1 + 2^i}\right) = R^2 Z^{2^{i + 1}} + S^2 Z^{2^{i + 1}} + \Gamma^{1 + 2^{i + 1}}.
    \]

    From the previous claim we know $[R, Z^{2^i}], [S, Z^{2^i}] \in \Gamma^{2^{i + 1}}$, so
    \begin{align*}
        \psi_i &\left(RZ^{2^i} + SZ^{2^i} + \Gamma^{1 + 2^i}\right) \\
        &= (RZ^{2^i})^2 + (SZ^{2^i})^2 + [RZ^{2^i}, SZ^{2^i}] + \Gamma^{1 + 2^{i + 1}} \\
        &= R^2 Z^{2^{i + 1}} + S^2 Z^{2^{i + 1}} + [R, S]Z^{2^{i + 1}} + \Gamma^{1 + 2^{i + 1}}.
    \end{align*}
    As $[R, S] \in \Gamma$, the claim follows.

    Now applying the last claim stepwise to $\psi_0$, $\psi_1$, \dots, $\psi_{\ell - 2}$ gives
    \begin{align*}
        &(\psi_{\ell - 2} \circ \dotsb \circ \psi_1 \circ \psi_0) (\lambda Z + \mu XZ + \nu YZ + TZ + \Gamma^2) \\
        &\qquad= \lambda^{2^{\ell - 1}} W + \mu^{2^{\ell - 1}} X^{2^{\ell - 1}}W + \nu^{2^{\ell - 1}} Y^{2^{\ell - 1}}W
        + T^{2^{\ell - 1}} W + \Gamma^{1 + 2^{\ell - 1}}.
    \end{align*}
    As $T^{2^{\ell - 1}} \in \Delta^{2^\ell}$, this means $T^{2^{\ell - 1}} W \in \Delta^d$.
    Thus applying $\pi$ eliminates the last term, and the lemma follows.
\end{proof}

\subsubsection{Specific cases}
Next we consider the coefficients $\alpha_A$ and $\alpha_B$.
As the calculations are no longer uniform for all groups, we introduce case distinctions.

\begin{lem}\label{lem:Alphas}
    If $G \cong \DG{3}$, $H \cong \DG{4}$ and $n > m > \ell$, then $\alpha_A = \alpha_B$.

    If $G \cong \DG{5}$ and $H \cong \DG{6}$, then we have the following.
    \begin{enumerate}
        \item $m > \ell$ implies $\alpha_A = \alpha_B$.
        \item $n = m = \ell$ implies $\alpha_A(\alpha_A + \beta_A) = \alpha_B(\alpha_B + \beta_B)$.
        \item $n = m < \ell$ implies $\alpha_A \beta_A = \alpha_B \beta_B = 0$.
    \end{enumerate}
\end{lem}
\begin{proof}
    Recall that $\varphi$ denotes the $2^m$-power map $\Delta/\Delta^2 \to \Delta^{2^m}/\Delta^{1 + 2^m}$.
    In this proof, the fact that the Frobenius map on $\F$ is injective will be used without further comment.

    First consider $G \cong \DG{5}$ and $H \cong \DG{6}$.
    By \cref{lem:FormulaPowerMap} and $Y^{2^m} = 0$, we have
    \[
        \varphi(A + \Delta^2) = \alpha_A^{2^m}X^{2^m} + (\alpha_A\beta_A)^{2^{m - 1}}Z^{2^{m - 1}} + \Delta^{1 + 2^m}
    \]
    and, if additionally $n = m$ holds,
    \[
        \varphi(A + \Delta^2) = \alpha_A^{2^m}W + (\alpha_A\beta_A)^{2^{m - 1}}Z^{2^{m - 1}} + \Delta^{1 + 2^m}.
    \]
    Similar expressions hold for $\varphi(B + \Delta^2)$.
    Note that the relation $a^{2^m} = b^{2^m}$,
    which holds in $H$, implies $\varphi(A + \Delta^2) = \varphi(B + \Delta^2)$.

    Now, consider $m > \ell$.
    Then $Z^{2^{m - 1}} = 0$, and so $\varphi(A + \Delta^2) = \varphi(B + \Delta^2)$ implies
    $\alpha_A^{2^m}X^{2^m} + \Delta^{1 + 2^m} = \alpha_B^{2^m} X^{2^m} + \Delta^{1 + 2^m}$.
    As $X^{2^m} \in \FD_{2^m} \smallsetminus \FD_{1 + 2^m}$, we have $\alpha_A = \alpha_B$.
    Note that if additionally $n = m$ holds, then we have $X^{2^m} = W$, but this does not change the conclusion,
    as then $W \in \FD_{2^m} \smallsetminus \FD_{1 + 2^m}$.

    Next, consider $n = m = \ell$.
    Then $\varphi(A + \Delta^2) = \varphi(B + \Delta^2)$ implies
    \[
        (\alpha_A^{2^m} + (\alpha_A \beta_A)^{2^{m - 1}})W + \Delta^{1 + 2^m} =  (\alpha_B^{2^m} + (\alpha_B \beta_B)^{2^{m - 1}})W + \Delta^{1 + 2^m}.
    \]
    As $W \in \FD_{2^m} \smallsetminus \FD_{1 + 2^m}$,
    this implies $\alpha_A^2 + \alpha_A\beta_A = \alpha_B^2 + \alpha_B\beta_B$ and hence the claim.

    Finally, consider $n = m < \ell$.
    Then $W \in \Delta^{1 + 2^m}$
    and $C^{2^{\ell - 1}} \in \Delta^{1 + 2^m}$.
    So because of $a^{2^m} = b^{2^m} = c^{2^{\ell - 1}}$, we have
    \[
        \varphi(A + \Delta^2) = \varphi(B + \Delta^2) = 0 + \Delta^{1 + 2^m}.
    \]
    As
    \begin{align*}
        \varphi(A + \Delta^2) &= (\alpha_A \beta_A)^{2^{m - 1}}Z^{2^{m - 1}} + \Delta^{1 + 2^m}, \\
        \varphi(B + \Delta^2) &= (\alpha_B \beta_B)^{2^{m - 1}}Z^{2^{m - 1}} + \Delta^{1 + 2^m},
    \end{align*}
    this together with $Z^{2^{m - 1}} \in \FD_{2^m} \smallsetminus \FD_{1 + 2^m}$ implies $\alpha_A \beta_A = \alpha_B \beta_B = 0$.

    Finally, consider $G \cong \DG{3}$, $H \cong \DG{4}$ and $n > m > \ell$.
    Then by \cref{lem:FormulaPowerMap} and using $Y^{2^m} = W$, $Z^{2^\ell} = 0$ and $m > \ell$ we get from \eqref{eq:AB}:
    \begin{equation}\label{eq:PhiAB}
        \begin{aligned}
            \varphi(A + \Delta^2) &= \alpha_A^{2^m} X^{2^m} + \beta_A^{2^m} W + \Delta^{1 + 2^m}, \\
            \varphi(B + \Delta^2) &= \alpha_B^{2^m} X^{2^m} + \beta_B^{2^m} W + \Delta^{1 + 2^m}.
        \end{aligned}
    \end{equation}
    Now $H \cong \DG{4}$ means $b^{2^m} = a^{2^m}c^{2^{\ell - 1}}$, so
    \begin{equation}\label{eq:PhiBG4}
        \begin{aligned}
            \varphi(B + \Delta^2)
            &= B^{2^m} + \Delta^{1 + 2^m} = 1 + a^{2^m}c^{2^{\ell - 1}} + \Delta^{1 + 2^m} \\
            &= A^{2^m} + C^{2^{\ell - 1}} + A^{2^m} C^{2^{\ell - 1}} + \Delta^{1 + 2^m}.
        \end{aligned}
    \end{equation}
    We next consider $A^{2^m}$ and $B^{2^m}$ modulo $\Gamma + \Delta^{1 + 2^m}$ which is done by ``deleting the $Z$'s and $C$'s'', roughly speaking.
    Formally, \eqref{eq:PhiAB} and \eqref{eq:PhiBG4} imply
    \begin{align*}
        \alpha_A^{2^m} X^{2^m}
        \equiv A^{2^m}
        \equiv B^{2^m}
        \equiv \alpha_B^{2^m} X^{2^m}
        \mod \Gamma + \Delta^{1 + 2^m}.
    \end{align*}
    Now as $X^{2^m}$ does not lie in $\Gamma + \Delta^{1 + 2^m}$ by \eqref{eq:BasisGamma} and \eqref{eq:BasisDelta},
    we get $\alpha_A = \alpha_B$.
\end{proof}

With these preparations, we are finally ready to answer the modular isomorphism problem for all the remaining cases.
In the first two cases, $m \le \ell$ holds, and we use the explicit approximations of $A^{2^n}$ and $B^{2^n}$ given by \cref{lem:A2NB2N}.
In the last two cases, $m > \ell$ holds, and we use the explicit approximation of $C^{2^{\ell - 1}}$ given by \cref{lem:CModulo,lem:Psi}.

\begin{lem}\label{lem:NeqMltLG5vsG6}
    If $n = m < \ell$, then $\F \DG{5} \not\cong \F \DG{6}$.
\end{lem}
\begin{proof}
    Assume the notation \eqref{eq:AB}, $n = m < \ell$, $G \cong \DG{5}$ and $H \cong \DG{6}$.
    By the relations $X^{2^n} = W$ and $Y^{2^n} = 0$, we have
    \[
        A^{2^n} \equiv \alpha_A^{2^n} W \mod \Z(\F G) \cap [\F G, \F G]
    \]
    by \cref{lem:A2NB2N}.
    Similarly, we have
    \[
        B^{2^n} \equiv \alpha_B^{2^n} W \mod \Z(\F G) \cap [\F G, \F G].
    \]
    As in $H$ the relation $a^{2^n} = b^{2^n}$ holds by $n = m$,
    from the fact that $W$ is a non-zero element in $\F\Z(G)$
    and the decomposition $\Z(\F G) = \F\Z(G) \oplus (\Z(\F G) \cap [\F G, \F G])$, we conclude that $\alpha_A = \alpha_B$.

    As $n = m < \ell$, \cref{lem:Alphas} implies $\alpha_A \beta_A = \alpha_B \beta_B = 0$.
    Hence, as $\alpha_A = \alpha_B$, we get $\alpha_A \beta_B + \alpha_B \beta_A = 0$, which contradicts \eqref{eq:Determinant}.
\end{proof}

The proof of the next case parallels that of the first case.

\begin{lem}\label{lem:NeqMeqLG5vsG6}
    If $n = m = \ell$, then $\F \DG{5} \not\cong \F \DG{6}$.
\end{lem}
\begin{proof}
    Assume the notation \eqref{eq:AB}, $n = m = \ell$, $G \cong \DG{5}$ and $H \cong \DG{6}$.
    We conclude that $\alpha_A = \alpha_B$ by the same arguments as in the proof of \cref{lem:NeqMltLG5vsG6}.

    As $n = m = \ell$, \cref{lem:Alphas} and $\alpha_A = \alpha_B$ imply $\alpha_A \beta_A = \alpha_B \beta_B$.
    Hence $\alpha_A \beta_B + \alpha_B \beta_A = 0$, which contradicts \eqref{eq:Determinant}.
\end{proof}

\begin{lem}\label{lem:MgtLG5vsG6}
    If $m > \ell$, then $\F \DG{5} \not\cong \F \DG{6}$.
\end{lem}
\begin{proof}
    Assume the notation \eqref{eq:AB}, $m > \ell$, $G \cong \DG{5}$ and $H \cong \DG{6}$.
    As $n > \ell$ we obtain $\Agemo_{n - 1}(\Z(\F G)) = \F(\Agemo_{n - 1}(\Z(G)))$ by \cref{prop:AgemoCenter}.
    This is generated by $W$ as an algebra since \cref{lem:Center} shows that $\Agemo_{n - 1}(\Z(G))$ is generated by $w$.
    Thus $\Agemo_{n - 1}(\Z(\F G)) \cap \Gamma = \F W$.
    As $a^2$ is central and $c^{2^{\ell - 1}} = a^{2^n}$ by the defining relations of $H$,
    we see that $C^{2^{\ell - 1}} = (A^2)^{2^{n - 1}}$ lies in $\Agemo_{n - 1}(\Z(\F G))$.
    Hence $C^{2^{\ell - 1}} \in \F W$.

    Recall that $\psi$ is defined as a $2^{\ell - 1}$-power map on $\Gamma/\Gamma^2$.
    Now by \cref{lem:CModulo,lem:Psi} we know
    \[
        \psi(C + \Gamma^2) =  \lambda^{2^{\ell - 1}}W + \mu^{2^{\ell - 1}}X^{2^{\ell - 1}}W + \nu^{2^{\ell - 1}}Y^{2^{\ell - 1}}W + \left(\Gamma^{1 + 2^{\ell - 1}} + \Delta^d\right),
    \]
    where $\lambda = \alpha_A \beta_B + \alpha_B \beta_A$, $\mu = \lambda(1 + \alpha_A + \alpha_B)$ and $\nu = \lambda(1 + \beta_A + \beta_B)$.

    As $X^{2^{\ell - 1}}W$ does not lie in $\F W$ and neither in $\Gamma^{1 + 2^{\ell - 1}} + \Delta^d$ by \eqref{eq:BasisGamma} and \eqref{eq:BasisDelta},
    the coefficient of it in the expression of $\psi(C + \Gamma^2)$ is $0$ by $C^{2^{\ell - 1}} \in \F W$.
    This coefficient equals a power of $\mu$ and by \cref{lem:Alphas}, this implies $\lambda = 0$, which contradicts \eqref{eq:Determinant}.
\end{proof}

The proof of the last case parallels that of the previous case.

\begin{lem}\label{lem:NgtMgtLG3vsG4}
    If $n > m > \ell$, then $\F \DG{3} \not\cong \F \DG{4}$.
\end{lem}
\begin{proof}
    Assume the notation \eqref{eq:AB}, $n > m > \ell$, $G \cong \DG{3}$ and $H \cong \DG{4}$.
    As $m > \ell$ we obtain $\Agemo_{m - 1}(\Z(\F G)) = \F(\Agemo_{m - 1}(\Z(G)))$ by \cref{prop:AgemoCenter}.
    This is generated by $X^{2^m}$ and $W$ as an algebra since \cref{lem:Center} shows that $\Agemo_{m - 1}(\Z(G))$ is generated by $x^{2^m}$ and $w$.
    Recall that $d = 1 + 2^{\ell - 1} + q$.
    Now, the weight of $X^{2^m}W$ is $2^m + q$ which is bigger than $d$ as $m > \ell$.
    Thus $\Agemo_{m - 1}(\Z(\F G)) \cap \Gamma \subseteq \F W + \Delta^d$.
    As $a^{-2}b^2$ is central and
    $(a^{-2}b^2)^{2^{m - 1}} = a^{-2^m}b^{2^m} = a^{-2^m}a^{2^m}c^{2^{\ell - 1}} = c^{2^{\ell - 1}}$
    by the defining relations of $H$,
    we see that $C^{2^{\ell - 1}} = (1 + a^{-2}b^2)^{2^{m - 1}}$ lies in $\Agemo_{m - 1}(\Z(\F G))$.
    Hence $C^{2^{\ell - 1}} \in \F W + \Delta^d$.

    The rest of the argument is essentially the same as at the end of the proof of the previous lemma:
    by the previous paragraph we known that the coefficient of $X^{2^{\ell - 1}}W$ in the expression of $\psi(C + \Gamma^2)$ is $0$.
    On the other hand, by \cref{lem:CModulo,lem:Psi}, this coefficient equals a power of $\mu$.
    By \cref{lem:Alphas} we conclude that $\lambda = 0$, which contradicts \eqref{eq:Determinant}.
\end{proof}

Now we give a proof of \cref{main:A} using the results obtained so far.

\begin{proof}[Proof of \cref{main:A}]
    It remains to show that the groups in \cref{cor:Counterexamples} are
    the only counterexamples to the modular isomorphism problem in our class.

    First we show that $H$ is two-generated and the central quotient of $H$ is dihedral.
    The minimal number of generators is a well-known invariant
    and a proof can be found in~\cite[Lemma 14.2.7]{Pas77}.
    Note that the central quotient of $H$ is dihedral if and only if $|H : \Z(H)| \ge 8$ and $|H : \D{H}\Z(H)| = 4$;
    thus it is also an invariant by a result of Margolis, Sakurai and Stanojkovski \cite[Corollary 2.8, Lemma 5.9]{MSS23}.
    Since the order and the isomorphism type of the abelianization are invariant, the groups have common parameters $n$, $m$ and $\ell$.

    By \cref{thm:Presentations,thm:IsomorphicAlgebras,prop:RemainingCases}, it remains to show that $\F \DG{5} \not\cong \F \DG{6}$
    when either $n = m \ge 2$ or $n > m > \ell$ and that $\F \DG{3} \not\cong \DG{4}$ when $n > m > \ell$.
    If $n = m \ge 2$, then we conclude that $\F \DG{5} \not\cong \F \DG{6}$ from \cref{lem:NeqMltLG5vsG6,lem:NeqMeqLG5vsG6,lem:MgtLG5vsG6}.
    If $n > m > \ell$, then we conclude that $\F \DG{5} \not\cong \F \DG{6}$ from \cref{lem:MgtLG5vsG6},
    and $\F \DG{3} \not\cong \F \DG{4}$ from \cref{lem:NgtMgtLG3vsG4}.
\end{proof}

\begin{rmk}\label{rmk:Assignment}
    The arguments involved in the group base approximation partly explain the choice we make for the isomorphism in \cref{thm:IsomorphicAlgebras}.
    There are of course many more choices for this isomorphism, but \cref{lem:Gammas} imposes some restrictions, for example.
    A similar argument, as in the first case proved in \cref{lem:Alphas}, can be used to show that $\alpha_A = \alpha_B$ needs to hold in the setting of \cref{thm:IsomorphicAlgebras}.
    Eventually the choice we make for the isomorphism seems to be the easiest.

    Moreover, the arguments in Section~\ref{sec:PositiveSolutions} suggest that
    the straightforward calculations in the proof of \cref{thm:IsomorphicAlgebras} not only make this proof easy, but also make it possible.
    Any relation in a group $H$ that requires to compute an actual commutator in the group algebra of another group $G$ for some given elements
    is not only difficult to verify by hand or even by computer;
    it imposes conditions that are difficult to meet in the first place.
\end{rmk}

\appendix
\section{Distinction of groups}\label{sec:Appendix}
In the class of two-generated finite $2$-groups with dihedral central quotient,
each group has a presentation of the form in \cref{thm:Presentations}.
The goal of this appendix is to establish by group-theoretical arguments when these groups are non-isomorphic
and to obtain the classification of groups within this class.
Let $n$, $m$ and $\ell$ denote positive integers.
Recall that a group is \emph{homocyclic} if it is isomorphic to a direct product of copies of a cyclic group.

\begin{thm}\label{thm:Hexachotomy}
    The six groups $\DG{1}$, \dots, $\DG{6}$ are pairwise non-isomorphic
    if they do not have homocyclic abelianizations \textup{(}i.e. $n > m$\textup{)}.
\end{thm}

\begin{thm}\label{thm:Trichotomy}
    The three groups $\DG{1}$, $\DG{5}$, $\DG{6}$ are pairwise non-isomorphic
    if they have homocyclic abelianizations \textup{(}i.e. $n = m$\textup{)}.
    Moreover, in this case $\DG{1} \cong \DG{2}$ and $\DG{3} \cong \DG{4} \cong \DG{5}$.
\end{thm}

These theorems follow immediately from \cref{main:A,lem:G1vsG2};
nevertheless, we will provide a direct group-theoretical proof.
We do this by considering the centers and maximal quotients of the groups.
We will write
\[
    \DG{1}(n, m, \ell), \dotsc, \DG{6}(n, m, \ell)
\]
to make the parameters $n$, $m$ and $\ell$ explicit when necessary.

\subsection{Maximal quotients}
To describe the maximal quotients, we first describe the socle in each case,
as the central involutions are exactly the generators of minimal normal subgroups.
By the socle of a finite $p$-group, we mean the subgroup generated by central elements of order~$p$.
From \cref{lem:Center}, we directly get the socles in all cases.

\begin{lem}\label{lem:Socle}
    Assume $n \ge 2$.
    Then the following holds.
    \begin{align*}
        \Soc{\DG{1}} &= \begin{cases}
            \langle x^{2^{n - 1}} \rangle \times \langle y^{2^{m - 1}} \rangle \times \langle z^{2^{\ell - 1}} \rangle                            \cong \Cyc{2} \times \Cyc{2} \times \Cyc{2}  & (m \ge 2) \\
            \langle x^{2^{n - 1}} \rangle \times \langle z^{2^{\ell - 1}} \rangle \phantom{\langle y^{2^{m - 1}} \rangle \times {}}               \cong \Cyc{2} \times \Cyc{2}                 & (m = 1);
        \end{cases}  \\
        \Soc{\DG{2}} &= \begin{cases}
            \langle x^{2^{n - 1}} \rangle \times \langle x^{-2^{m - 1}}y^{2^{m - 1}} \rangle \times \langle z^{2^{\ell - 1}} \rangle              \cong \Cyc{2} \times \Cyc{2} \times \Cyc{2}  & (m \ge 2) \\
            \langle x^{2^{n - 1}} \rangle \times \langle z^{2^{\ell - 1}} \rangle \phantom{\langle x^{-2^{m - 1}}y^{2^{m - 1}} \rangle \times {}} \cong \Cyc{2} \times \Cyc{2}                 & (m = 1);
        \end{cases}  \\
        \Soc{\DG{3}} &=
            \langle x^{2^{n - 1}} \rangle \times \langle z^{2^{\ell - 1}} \rangle                                                                 \cong \Cyc{2} \times \Cyc{2};                &           \\
        \Soc{\DG{4}} &=
            \langle x^{2^{n - 1}} \rangle \times \langle z^{2^{\ell - 1}} \rangle                                                                 \cong \Cyc{2} \times \Cyc{2};                &           \\
        \Soc{\DG{5}} &= \begin{cases}
            \langle y^{2^{m - 1}} \rangle \times \langle z^{2^{\ell - 1}} \rangle                                                                 \cong \Cyc{2} \times \Cyc{2}                 & (m \ge 2) \\
            \langle z^{2^{\ell - 1}} \rangle \phantom{\langle y^{2^{m - 1}} \rangle \times {}}                                                    \cong \Cyc{2}                                & (m = 1);
        \end{cases}  \\
        \Soc{\DG{6}} &= \begin{cases}
            \langle x^{-2^{m - 1}}y^{2^{m - 1}} \rangle \times \langle z^{2^{\ell - 1}} \rangle                                                   \cong \Cyc{2} \times \Cyc{2}                 & (m \ge 2) \\
            \langle z^{2^{\ell - 1}} \rangle \phantom{\langle x^{-2^{m - 1}}y^{2^{m - 1}} \rangle \times {}}                                      \cong \Cyc{2}                                & (m = 1).
        \end{cases}
    \end{align*}
\end{lem}

The next lemma describes all the maximal quotients of the six groups in \cref{thm:Presentations}.

\begin{lem}\label{lem:MaximalQuotients}
    Assume $n \ge 2$.
    Let $G$ be one of the groups $\DG{1}(n, m, \ell)$, \dots, $\DG{6}(n, m, \ell)$ and $Q$ a maximal quotient of $G$.
    Then $Q \cong G_{\scriptsizetextdice{\textnormal{?}}}(n', m', \ell')$
    for some $\textdice{\textnormal{?}} \in \{\dice{1}, \dotsc, \dice{6}\}$ and $n', m', \ell'$ satisfying $n'+m'+\ell' = n+m+\ell-1$\textup{;}
    we need to consider five cases
    \begin{enumerate}
        \item $n > m + 1$ and $m \ge 2$ \textup{(}\cref{tbl:MQCase1}\textup{)}
        \item $n = m + 1$ and $m \ge 2$ \textup{(}\cref{tbl:MQCase2}\textup{)}
        \item $n = m$     and $m \ge 2$ \textup{(}\cref{tbl:MQCase3}\textup{)}
        \item $n > m + 1$ and $m = 1$   \textup{(}\cref{tbl:MQCase4}\textup{)}
        \item $n = m + 1$ and $m = 1$   \textup{(}\cref{tbl:MQCase5}\textup{)}
    \end{enumerate}
    for all the possible values of the parameters for $Q$,
    and those are summarized in \cref{tbl:MQCase1,tbl:MQCase2,tbl:MQCase3,tbl:MQCase4,tbl:MQCase5}.

    \begin{table}[phtb]
        \centering

        \begin{tabular}{cccc}
            \toprule
                                 & $(n - 1, m, \ell)$ & $(n, m - 1, \ell)$ & $(n, m, \ell - 1)$ \\
            \midrule
            $\DG{1}(n, m, \ell)$ & $\DG{1}$ $\DG{5}$  & $\DG{1}$ $\DG{3}$  & $\DG{1}$ \\
            $\DG{2}(n, m, \ell)$ & $\DG{2}$ $\DG{6}$  & $\DG{2}$ $\DG{4}$  & $\DG{2}$ \\
            $\DG{3}(n, m, \ell)$ & $\DG{3}$ $\DG{5}$  &                    & $\DG{1}$ \\
            $\DG{4}(n, m, \ell)$ & $\DG{4}$ $\DG{6}$  &                    & $\DG{2}$ \\
            $\DG{5}(n, m, \ell)$ &                    & $\DG{5}$           & $\DG{1}$ \\
            $\DG{6}(n, m, \ell)$ &                    & $\DG{6}$           & $\DG{2}$ \\
            \bottomrule \\
        \end{tabular}
        \vspace*{-.3cm}
        \caption{Maximal quotients when $n > m + 1$, $m \ge 2$.}
        \label{tbl:MQCase1}
        \begin{tabular}{cccc}
            \toprule
                                 & $(n - 1, m, \ell)$ & $(n, m - 1, \ell)$ & $(n, m, \ell - 1)$ \\
            \midrule
            $\DG{1}(n, m, \ell)$ & $\DG{1}$ $\DG{5}$  & $\DG{1}$ $\DG{3}$  & $\DG{1}$ \\
            $\DG{2}(n, m, \ell)$ & $\DG{1}$ $\DG{6}$  & $\DG{2}$ $\DG{4}$  & $\DG{2}$ \\
            $\DG{3}(n, m, \ell)$ & $\DG{5}$ $\DG{6}$  &                    & $\DG{1}$ \\
            $\DG{4}(n, m, \ell)$ & $\DG{5}$           &                    & $\DG{2}$ \\
            $\DG{5}(n, m, \ell)$ &                    & $\DG{5}$           & $\DG{1}$ \\
            $\DG{6}(n, m, \ell)$ &                    & $\DG{6}$           & $\DG{2}$ \\
            \bottomrule \\
        \end{tabular}
        \vspace*{-.3cm}
        \caption{Maximal quotients when $n = m + 1$, $m \ge 2$.}
        \label{tbl:MQCase2}

        \begin{tabular}{cccc}
            \toprule
                                 & $(n - 1, m, \ell)$ & $(n, m - 1, \ell)$                  & $(n, m, \ell - 1)$ \\
            \midrule
            $\DG{1}(n, m, \ell)$ &                    & $\DG{1}$ $\DG{2}$ $\DG{3}$ $\DG{4}$ & $\DG{1}$ \\
            $\DG{2}(n, m, \ell)$ &                    & $\DG{1}$ $\DG{2}$ $\DG{3}$ $\DG{4}$ & $\DG{2}$ \\
            $\DG{3}(n, m, \ell)$ &                    & $\DG{5}$                            & $\DG{1}$ \\
            $\DG{4}(n, m, \ell)$ &                    & $\DG{5}$                            & $\DG{2}$ \\
            $\DG{5}(n, m, \ell)$ &                    & $\DG{5}$                            & $\DG{1}$ \\
            $\DG{6}(n, m, \ell)$ &                    & $\DG{6}$                            & $\DG{2}$ \\
            \bottomrule \\
        \end{tabular}
        \vspace*{-.3cm}
        \caption{Maximal quotients when $n = m$, $m \ge 2$.}
        \label{tbl:MQCase3}

        \begin{tabular}{cccc}
            \toprule
                                 & $(n - 1, m, \ell)$ & $(n, m - 1, \ell)$ & $(n, m, \ell - 1)$ \\
            \midrule
            $\DG{1}(n, m, \ell)$ & $\DG{1}$ $\DG{5}$  &                    & $\DG{1}$ \\
            $\DG{2}(n, m, \ell)$ & $\DG{2}$ $\DG{6}$  &                    & $\DG{2}$ \\
            $\DG{3}(n, m, \ell)$ & $\DG{3}$ $\DG{5}$  &                    & $\DG{1}$ \\
            $\DG{4}(n, m, \ell)$ & $\DG{4}$ $\DG{6}$  &                    & $\DG{2}$ \\
            $\DG{5}(n, m, \ell)$ &                    &                    & $\DG{1}$ \\
            $\DG{6}(n, m, \ell)$ &                    &                    & $\DG{2}$ \\
            \bottomrule \\
        \end{tabular}
        \vspace*{-.3cm}
        \caption{Maximal quotients when $n > m + 1$, $m = 1$.}
        \label{tbl:MQCase4}
        \begin{tabular}{cccc}
            \toprule
                                 & $(n - 1, m, \ell)$ & $(n, m - 1, \ell)$ & $(n, m, \ell - 1)$ \\
            \midrule
            $\DG{1}(n, m, \ell)$ & $\DG{1}$ $\DG{5}$  &                    & $\DG{1}$ \\
            $\DG{2}(n, m, \ell)$ & $\DG{1}$ $\DG{6}$  &                    & $\DG{2}$ \\
            $\DG{3}(n, m, \ell)$ & $\DG{5}$ $\DG{6}$  &                    & $\DG{1}$ \\
            $\DG{4}(n, m, \ell)$ & $\DG{5}$           &                    & $\DG{2}$ \\
            $\DG{5}(n, m, \ell)$ &                    &                    & $\DG{1}$ \\
            $\DG{6}(n, m, \ell)$ &                    &                    & $\DG{2}$ \\
            \bottomrule \\
        \end{tabular}
        \vspace*{-.3cm}
        \caption{Maximal quotients when $n = m + 1$, $m = 1$.}
        \label{tbl:MQCase5}
    \end{table}
\end{lem}

\begin{proof}
    In all cases, the proof is obtained by applying \cref{lem:Reduction} to each maximal quotient in each case,
    which corresponds to a minimal normal subgroup generated by a non-trivial element of the socle in \cref{lem:Socle}.
\end{proof}

We finish this subsection with an illustration of the studied groups by a graph, which can be derived from \cref{lem:MaximalQuotients} in \cref{fig:Forest}.

\begin{figure}[phtb]
    \centering
    \rotatebox{90}{\begin{minipage}{0.95\textheight}
        \centering
        \includegraphics[width=0.90\textwidth]{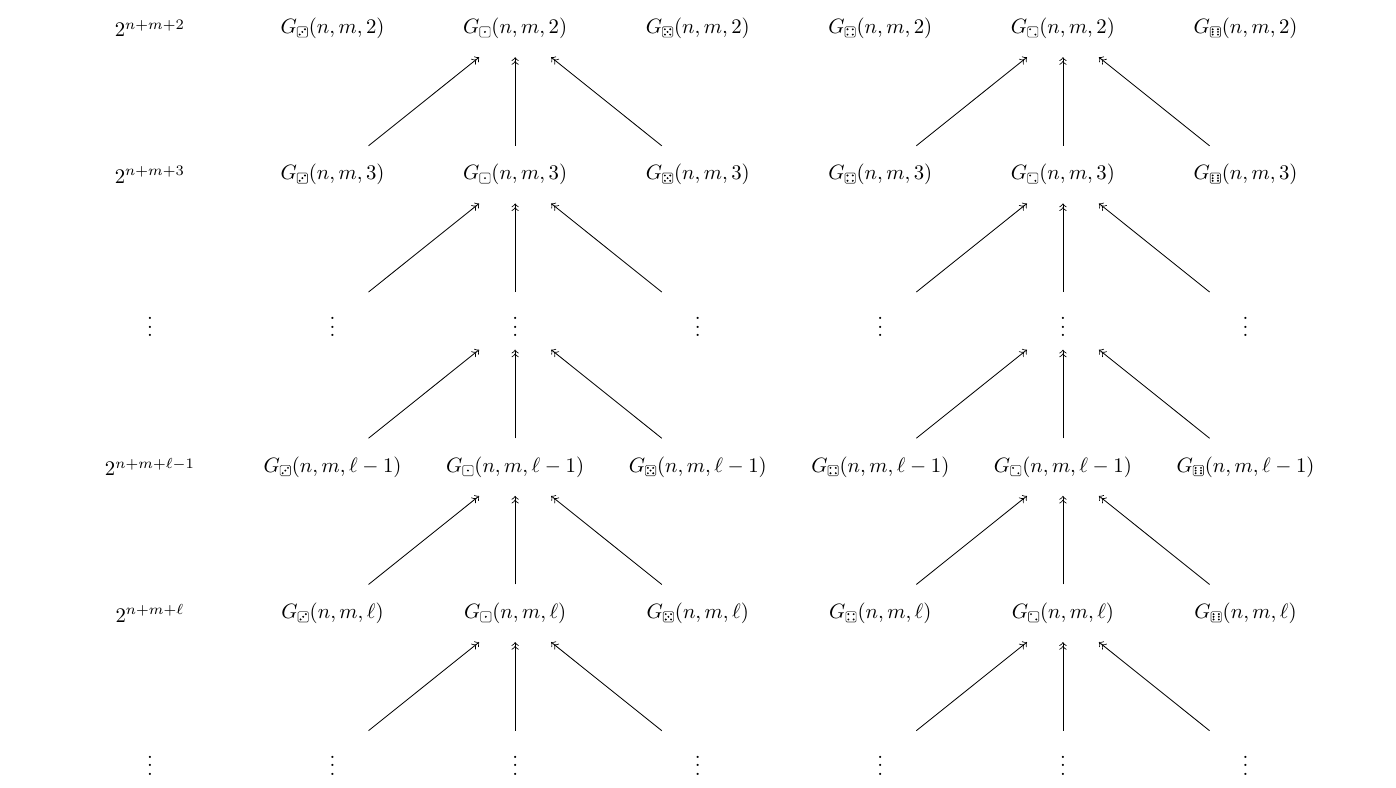}
        \caption{
            Finite $2$-groups with dihedral central quotient whose abelianization is isomorphic to $\Cyc{2^n} \times \Cyc{2^m}$ and not homocyclic.
            An arrow from a group $G$ to $H$ is drawn if $H$ is isomorphic to a maximal quotient of $G$,
            that is, $G$ is an immediate descendant of $H$ .
        }
        \label{fig:Forest}
    \end{minipage}}
\end{figure}

\subsection{Inductive argument}
We will now proceed to prove \cref{thm:Hexachotomy,thm:Trichotomy}.
This will be achieved by a double induction.
The outer induction runs on the first parameter, i.e. $n$,
and the inner on the second parameter, i.e. $m$.
We will need the following two cases for the base of the inductions.

Recall the standard presentations of finite $2$-groups of maximal class:
the dihedral, semidihedral and generalized quaternion groups.
\begin{alignat*}{3}
    D_{2^{n + 1}} = \langle\, a, b \mid a^2 &= 1,             &\quad b^{2^n} &= 1, &\quad b^a &= b^{-1} \,\rangle; \\
    S_{2^{n + 1}} = \langle\, a, b \mid a^2 &= 1,             &\quad b^{2^n} &= 1, &\quad b^a &= b^{2^{n - 1} - 1} \,\rangle; \\
    Q_{2^{n + 1}} = \langle\, a, b \mid a^2 &= b^{2^{n - 1}}, &\quad b^{2^n} &= 1, &\quad b^a &= b^{-1} \,\rangle.
\end{alignat*}

\begin{prop}\label{prop:MaximalClass}
    Assume $\ell \ge 2$.
    Then
    \begin{align*}
        D_{2^{\ell + 2}} \cong \DG{1}(1, 1, \ell); \quad
        S_{2^{\ell + 2}} \cong \DG{5}(1, 1, \ell); \quad
        Q_{2^{\ell + 2}} \cong \DG{6}(1, 1, \ell).
    \end{align*}
    In particular, $\DG{1}(1, 1, \ell)$, $\DG{5}(1, 1, \ell)$ and $\DG{6}(1, 1, \ell)$ are pairwise non-isomorphic.
\end{prop}
\begin{proof}
    Since the three cases have the same kind of isomorphism $a \mapsto y$ and $b \mapsto xy$,
    we only prove that $Q_{2^{\ell + 2}} \cong \DG{6}(1, 1, \ell)$ here.

    Let $\DG{6}(1, 1, \ell) = \langle x, y, z \rangle$ and define $a = y$ and $b = xy$.
    Since
    \begin{align*}
        b^2
        &= xyxy
        = x^2 y[y, x]y
        = x^2 y zy \\
        &= z^{2^{\ell - 1}} y zy
        = z^{2^{\ell - 1}} y^2 z[z, y]
        = z^{2^\ell} z[z, y]
        = z^{-1},
    \end{align*}
    we have, using that $y^2$ is central in the last equation,
    \begin{align*}
        a^2              &= y^2 = z^{2^{\ell - 1}} = z^{-2^{\ell - 1}} = b^{2^\ell}; \\
        b^{2^{\ell + 1}} &= z^{-2^{\ell}} = 1; \\
        b^a              &= (xy)^y = y^{-1} x y^2 = yx = xy[y, x] = xyz = b b^{-2} = b^{-1}.
    \end{align*}
    Hence there is an epimorphism from $Q_{2^{\ell + 2}}$ to $\DG{6}(1, 1, \ell)$.
    Since both groups have the same order, they must be isomorphic.
\end{proof}

\begin{lem}\label{lem:CornerCases}
    Assume $n \ge 2$.
    Then $\DG{1}(n, 1, 1) \not\cong \DG{2}(n, 1, 1)$.
\end{lem}
\begin{proof}
    Observe that $\DG{1}(n, 1, 1)$ has a maximal subgroup that is isomorphic to $\Cyc{2^{n - 1}} \times \Cyc{2} \times \Cyc{2}$,
    namely $\langle x^2, y, z \rangle$, while $\DG{2}(n, 1, 1)$ does not.
    Indeed, we have $\Frat(\DG{2}(n, 1, 1)) = \langle x^2, z \rangle$,
    so that the maximal subgroups of $\DG{2}(n, 1, 1)$ are $\langle x, z \rangle$, $\langle y, z \rangle$ and $\langle xy, x^2 \rangle$.
    Note here that $(xy)^2 = x^2 y^2 z = x^4 z$, so that $z \in \langle xy, x^2 \rangle$.
    Hence all the maximal subgroups of $\DG{2}(n, 1, 1)$ are two-generated.
\end{proof}

We are ready for the inductions.

\begin{proof}[Proof of \cref{thm:Hexachotomy,thm:Trichotomy}]
    We first note that when $n = m$, then the isomorphisms
    $\DG{1}(n, m, \ell) \cong \DG{2}(n, m, \ell)$ and $\DG{3}(n, m, \ell) \cong \DG{4}(n, m, \ell) \cong \DG{5}(n, m, \ell)$
    are clear from the presentations.

    We proceed by induction on $n$.
    If $n = 1$, then also $m = 1$ and the groups $\DG{1}(n, m, \ell)$, $\DG{5}(n, m, \ell)$ and $\DG{6}(n, m, \ell)$ are pairwise non-isomorphic by \cref{prop:MaximalClass}.

    So assume that $n \ge 2$ and that \cref{thm:Hexachotomy,thm:Trichotomy} hold for smaller values of $n$.
    We proceed by induction on $m$.
    When $m = 1$, then by \cref{lem:Center} the centers of $\DG{1}(n, m, \ell)$, $\DG{2}(n, m, \ell)$, $\DG{3}(n, m, \ell)$ and $\DG{4}(n, m, \ell)$ are not isomorphic to
    the centers of $\DG{5}(n, m, \ell)$ and $\DG{6}(n, m, \ell)$.
    By induction and \cref{tbl:MQCase4,tbl:MQCase5}, only $\DG{1}(n, m, \ell)$ or $\DG{2}(n, m, \ell)$ can map onto
    one of the groups $\DG{1}(n - 1, m, \ell)$ or $\DG{2}(n - 1, m, \ell)$,
    so these two groups are not isomorphic to either of $\DG{3}(n, m, \ell)$ or $\DG{4}(n, m, \ell)$.
    To see that $\DG{1}(n, m, \ell)$ is not isomorphic to $\DG{2}(n, m, \ell)$ note that,
    again by induction and \cref{tbl:MQCase4,tbl:MQCase5},
    $\DG{2}(n, m, \ell)$ maps onto $\DG{6}(n - 1, m, \ell)$, while $\DG{1}(n, m, \ell)$ does not.
    To see that $\DG{3}(n, m, \ell)$ is not isomorphic to $\DG{4}(n, m, \ell)$ note that, by \cref{tbl:MQCase4,tbl:MQCase5},
    exactly one of them maps onto $\DG{6}(n - 1, m, \ell)$ by induction.
    To distinguish between $\DG{5}(n, m, \ell)$ and $\DG{6}(n, m, \ell)$ we observe that by \cref{tbl:MQCase4,tbl:MQCase5}
    the only maximal quotient of $\DG{5}(n, m, \ell)$ is $\DG{1}(n, m, \ell - 1)$, while the only maximal quotient of $\DG{6}(n, m, \ell)$ is $\DG{2}(n, m, \ell - 1)$.
    If $\ell = 2$, then these quotients are not isomorphic and hence neither are the groups $\DG{5}(n, m, \ell)$ and $\DG{6}(n, m, \ell)$ by \cref{lem:CornerCases}.
    If $\ell > 2$, then the quotients $\DG{1}(n, m, \ell - 1)$ and $\DG{2}(n, m, \ell - 1)$ are also not isomorphic, as we showed earlier in this paragraph.
    This finishes the case $m = 1$.

    Consider next $n > m \ge 2$ and assume that \cref{thm:Hexachotomy,thm:Trichotomy} hold for smaller values of $n$ or $m$.
    Looking on the centers of the groups in \cref{lem:Center},
    we see that we only need to show that
    $\DG{1}(n, m, \ell) \not\cong \DG{2}(n, m, \ell)$, $\DG{3}(n, m, \ell) \not\cong \DG{4}(n, m, \ell)$ and $\DG{5}(n, m, \ell) \not\cong \DG{6}(n, m, \ell)$.
    To distinguish between $\DG{1}(n, m, \ell)$ and $\DG{2}(n, m, \ell)$ note that by induction and \cref{tbl:MQCase1,tbl:MQCase2}
    the group $\DG{1}(n, m, \ell)$ maps onto $\DG{1}(n, m - 1, \ell)$, while $\DG{2}(n, m, \ell)$ does not.
    To see that $\DG{3}(n, m, \ell) \not\cong \DG{4}(n, m, \ell)$ we note that, again by induction and \cref{tbl:MQCase1,tbl:MQCase2},
    exactly one of them maps onto $\DG{6}(n - 1, m, \ell)$.
    Finally, to observe that $\DG{5}(n, m, \ell)$ and $\DG{6}(n, m, \ell)$ are not isomorphic we also use induction and \cref{tbl:MQCase1,tbl:MQCase2}
    and see that $\DG{5}(n, m, \ell)$ maps onto $\DG{5}(n, m - 1, \ell)$, while $\DG{6}(n, m, \ell)$ does not.

    The last case to consider is $n = m$.
    Recall that then $\DG{1}(n, m, \ell) \cong \DG{2}(n, m, \ell)$ and $\DG{3}(n, m, \ell) \cong \DG{4}(n, m, \ell) \cong \DG{5}(n, m, \ell)$.
    So after looking at the centers in \cref{lem:Center}, it only remains to show that $\DG{5}(n, m, \ell)$ and $\DG{6}(n, m, \ell)$ are not isomorphic.
    By induction and \cref{tbl:MQCase3}, we observe that $\DG{6}(n, m, \ell)$ maps onto $\DG{6}(n, m - 1, \ell)$ which is not the case for $\DG{5}(n, m, \ell)$.
    This finishes the inner induction and hence also the induction step of the outer induction.
\end{proof}

\textbf{Acknowledgments:}
This work has been supported by the Madrid Government (Comunidad de Madrid - Spain) under the multiannual Agreement with UAM in the line for the Excellence of the University Research Staff in the context of the V PRICIT (Regional Programme of Research and Technological Innovation).
The authors acknowledges financial support from the Spanish Ministry of Science and Innovation, through the Severo Ochoa Programme for Centers of Excellence in R\& D (CEX2019-000904-S) and through the Ramon y Cajal grant program.

\bibliographystyle{abbrv}
\bibliography{references}
\end{document}